\documentclass[12pt,a4paper]{article} 

\usepackage{amsthm}
\usepackage{amsfonts}
\usepackage{amsmath}
\usepackage{amssymb}
\usepackage{hyperref}
\usepackage[utf8]{inputenc}
\usepackage{breqn}
\usepackage{cleveref}
\crefname{equation}{formula}{formulas}
\usepackage{bm}
\usepackage[affil-it]{authblk}
\usepackage{etoolbox}
\patchcmd{\thebibliography}{\section*}{\section}{}{}
\usepackage{nicefrac}

\usepackage{ifthen}
\usepackage{amssymb}
\usepackage{graphicx}
\usepackage{scalefnt}
\usepackage{accents}
\newtheorem{thm}{Theorem}[section]
\newtheorem{lma}[thm]{Lemma}
\newtheorem*{lma*}{Lemma}
\newtheorem{crl}[thm]{Corollary}
\newtheorem{prop}[thm]{Proposition}

\newtheoremstyle{break}
  {\topsep}
  {\topsep}
  {\itshape}
  {}
  {\bfseries}
  {.}
  {\newline}
  {}
\theoremstyle{break}
\newtheorem{thmB}[thm]{Theorem}
\newtheorem{lmaB}[thm]{Lemma}
\newtheorem*{lmaB*}{Lemma}

\theoremstyle{definition}
\newtheorem{dfn}[thm]{Definition}
\newtheorem{rmrk}[thm]{Remark}
\newtheorem{qst}[thm]{Question}
\newtheorem{fact}[thm]{Fact}

\newtheoremstyle{breakDefinition}
  {\topsep}
  {\topsep}
  {}
  {}
  {\bfseries}
  {.}
  {\newline}
  {}
\theoremstyle{breakDefinition}
\newtheorem{factB}[thm]{Fact}

\newtheorem{rmrkB}[thm]{Remark}

\newcommand{\proofAVName}{Claim}
\newenvironment{proofAV}[1][\proofAVName]{\begin{proof}[Proof of the #1]\renewcommand*{\qedsymbol}{\(\scalebox{0.8}{$\blacksquare_{\hspace*{3pt}#1}$}\)}}{\end{proof}}
\newcommand*{\card}[1][M]{\vert #1\vert}

\DeclareMathOperator*{\bigdoublewedge}{\bigwedge\mkern-15mu\bigwedge}

\newcommand\existsScaled[1]{\vcenter{\hbox{\scalefont{#1}$\exists$}}}

\mathchardef\mhyphen="2D

\newcommand*{\frmI}[3][\phi]{#1_{#3}(\overline{#2})}

\newcommand*{\frmDblI}[5][\psi]{
	\ifthenelse{\equal{#4}{}}
		{#1_{#5}(\overline{#2},#3)}
		{#1_{#5}(\overline{#2},#3\overline{#4})}
}

\newcommand*{\frmIExTup}[5][\psi]{
	\ifthenelse{\equal{#4}{}}
		{\exists #3#1_{#5}(\overline{#2},#3)}
		{\exists #3\overline{#4}#1_{#5}(\overline{#2},#3\overline{#4})}
}

\newcommand*{\diagI}[3][A]{\diag_{#1_{#3}}(\overline{#2})}

\newcommand{\forkindepPrv}[1]{
  \mathrel{
    \mathop{
      \vcenter{
        \hbox{\oalign{\noalign{\kern-.3ex}\hfil$\vert$\hfil\cr
              \noalign{\kern-.7ex}
              $\smile$\cr\noalign{\kern-.3ex}}}
      }
    }\displaylimits_{#1}
  }
}

\newcommand{\pctext}[2]{\text{\parbox{#1}{\centering #2}}}

\DeclareMathOperator{\theory}{Th}
\DeclareMathOperator{\age}{Age}
\DeclareMathOperator{\cl}{cl}
\DeclareMathOperator{\sem}{sem}
\DeclareMathOperator{\type}{tp}

\newcommand*{\classPosit}[2][C]{\mathcal{#1}_{#2}^{>0}}
\newcommand*{\classPositOpr}[2][C]{(\mathcal{#1}_{#2}^{>0},\leq^{*})}

\newcommand*{\classPositEq}[2][C]{\mathcal{#1}_{#2}^{\geq 0}}
\newcommand*{\classPositEqOpr}[2][C]{(\mathcal{#1}_{#2}^{\geq 0},\leq)}

\newcommand*{\leqst}[3][\leq]{#2#1^{*}#3}

\newcommand{\minPair}[3][\leq]{#2\not{#1}_{min}#3}

\newcommand*{\intCl}[3][\leq^{*}]{#2#1_{i}#3}

\newcommand*{\closure}[3][*]{\cl^{#1}_{#2}(#3)}

\newcommand*{\freeJoin}[3]{#1\sqcup_{#2}#3}

\newcommand*{\notOldCl}[3]{#1\not{\overline{\sqcup}}_{#2}#3}
\newcommand*{\oldClLiteral}[3]{
	\ifthenelse{\equal{#1}{}}
		{\hspace{-3pt}old over $#2$ with respect to \hspace{-2pt}$#3$\hspace{-3pt}}
		{\hspace{-3pt}$#1$ is old over $#2$ with respect to \hspace{-2pt}$#3$\hspace{-3pt}}
}
\newcommand*{\notOldClLiteral}[4]{
	\ifthenelse{\equal{#4}{p}}
		{\ifthenelse{\equal{#1}{}}
			{\hspace{-3pt}new over $#2$ with respect to \hspace{-2pt}$#3$\hspace{-3pt}}
			{\hspace{-3pt}$#1$ is new over $#2$ with respect to \hspace{-2pt}$#3$\hspace{-3pt}}
		}
		{\ifthenelse{\equal{#1}{}}
			{\hspace{-3pt}not old over $#2$ with respect to \hspace{-2pt}$#3$\hspace{-3pt}}
			{\hspace{-3pt}$#1$ is not old over $#2$ with respect to \hspace{-2pt}$#3$\hspace{-3pt}}
		}
}

\newcommand*{\diagCl}[4][A]{\diag_{(#1,#2)}(#3,#4)}

\newcommand*{\diagClIExTupLParenth}[8][A]
	{\ifthenelse{\equal{#8}{B}}
		{\ifthenelse{\equal{#6}{}}
			{\ifthenelse{\equal{#7}{}}
				{\exists #5\Big( \diag_{(#1,#2)}(\overline{#3}#4,#5)\wedge}
				{\exists #5_{#7}\Big( \diag_{(#1,#2_{#7})}(\overline{#3}#4,#5_{#7})\wedge}
			}
			{\ifthenelse{\equal{#7}{}}
				{\exists #5\overline{#6}\Big( \diag_{(#1,#2)}(\overline{#3}#4,#5\overline{#6})\wedge}
				{\exists #5\overline{#6}_{#7}\Big( \diag_{(#1,#2_{#7})}(\overline{#3}#4,#5\overline{#6}_{#7})\wedge}
			}
		}
		{\ifthenelse{\equal{#8}{n}}
			{\ifthenelse{\equal{#6}{}}
				{\ifthenelse{\equal{#7}{}}
					{\exists #5( \diag_{(#1,#2)}(\overline{#3}#4,#5)\wedge}
					{\exists #5_{#7}( \diag_{(#1,#2_{#7})}(\overline{#3}#4,#5_{#7})\wedge}
				}
				{\ifthenelse{\equal{#7}{}}
					{\exists #5\overline{#6}( \diag_{(#1,#2)}(\overline{#3}#4,#5\overline{#6})\wedge}
					{\exists #5\overline{#6}_{#7}( \diag_{(#1,#2_{#7})}(\overline{#3}#4,#5\overline{#6}_{#7})\wedge}
				}
			}
			{\ifthenelse{\equal{#8}{gg}}
				{\ifthenelse{\equal{#6}{}}
					{\ifthenelse{\equal{#7}{}}
						{\exists #5\bigg( \diag_{(#1,#2)}(\overline{#3}#4,#5)\wedge}
						{\exists #5_{#7}\bigg( \diag_{(#1,#2_{#7})}(\overline{#3}#4,#5_{#7})\wedge}
					}
					{\ifthenelse{\equal{#7}{}}
						{\exists #5\overline{#6}\bigg( \diag_{(#1,#2)}(\overline{#3}#4,#5\overline{#6})\wedge}
						{\exists #5\overline{#6}_{#7}\bigg( \diag_{(#1,#2_{#7})}(\overline{#3}#4,#5\overline{#6}_{#7})\wedge}
					}
				}
				{\ifthenelse{\equal{#6}{}}
					{\ifthenelse{\equal{#7}{}}
						{\exists #5 \diag_{(#1,#2)}(\overline{#3}#4,#5)}
						{\exists #5_{#7} \diag_{(#1,#2_{#7})}(\overline{#3}#4,#5_{#7})}
					}
					{\ifthenelse{\equal{#7}{}}
						{\exists #5\overline{#6}\diag_{(#1,#2)}(\overline{#3}#4,#5\overline{#6})}
						{\exists #5\overline{#6}_{#7}\diag_{(#1,#2_{#7})}(\overline{#3}#4,#5\overline{#6}_{#7})}
					}
				}
			}
		}
	}
	
\newcommand*{\diagClIAllTupLParenth}[8][A]
	{\ifthenelse{\equal{#8}{B}}
		{\ifthenelse{\equal{#5}{}}
			{\ifthenelse{\equal{#7}{}}
				{\forall\overline{#6}\Big( \diag_{(#1,#2)}(\overline{#3}#4,\overline{#6})\rightarrow}
				{\forall\overline{#6}_{#7}\Big( \diag_{(#1,#2_{#7})}(\overline{#3}#4,\overline{#6}_{#7})\rightarrow}
			}
			{\ifthenelse{\equal{#7}{}}
				{\forall\overline{#6}\Big( \diag_{(#1,#2)}(\overline{#3}#4\overline{#5},\overline{#6})\rightarrow}
				{\forall\overline{#6}_{#7}\Big( \diag_{(#1,#2_{#7})}(\overline{#3}#4\overline{#5},\overline{#6}_{#7})\rightarrow}
			}
		}
		{\ifthenelse{\equal{#8}{n}}
			{\ifthenelse{\equal{#5}{}}
				{\ifthenelse{\equal{#7}{}}
					{\forall\overline{#6}(\diag_{(#1,#2)}(\overline{#3}#4,\overline{#6})\rightarrow}
					{\forall\overline{#6}_{#7}(\diag_{(#1,#2_{#7})}(\overline{#3}#4,\overline{#6}_{#7})\rightarrow}
				}
				{\ifthenelse{\equal{#7}{}}
					{\forall\overline{#6}(\diag_{(#1,#2)}(\overline{#3}#4\overline{#5},\overline{#6})\rightarrow}	
					 {\forall\overline{#6}_{#7}(\diag_{(#1,#2_{#7})}(\overline{#3}#4\overline{#5},\overline{#6}_{#7})\rightarrow}
				}
			}
			{\ifthenelse{\equal{#8}{neg}}
				{\ifthenelse{\equal{#5}{}}
					{\ifthenelse{\equal{#7}{}}
						{\forall\overline{#6}\neg \diag_{(#1,#2)}(\overline{#3}#4,\overline{#6})}
						{\forall\overline{#6}_{#7}\neg \diag_{(#1,#2_{#7})}(\overline{#3}#4,\overline{#6}_{#7})}
					}
					{\ifthenelse{\equal{#7}{}}
						{\forall\overline{#6}\neq \diag_{(#1,#2)}(\overline{#3}#4\overline{#5},\overline{#6})}
						{\forall\overline{#6}_{#7}\neq \diag_{(#1,#2_{#7})}(\overline{#3}#4\overline{#5},\overline{#6}_{#7})}
					}
				}
				{\ifthenelse{\equal{#8}{gg}}
					{\ifthenelse{\equal{#5}{}}
						{\ifthenelse{\equal{#7}{}}
							{\forall\overline{#6}\bigg( \diag_{(#1,#2)}(\overline{#3}#4,\overline{#6})\rightarrow}
							{\forall\overline{#6}_{#7}\bigg( \diag_{(#1,#2_{#7})}(\overline{#3}#4,\overline{#6}_{#7})\rightarrow}
						}
						{\ifthenelse{\equal{#7}{}}
							{\forall\overline{#6}\bigg( \diag_{(#1,#2)}(\overline{#3}#4\overline{#5},\overline{#6})\rightarrow}
							{\forall\overline{#6}_{#7}\bigg( \diag_{(#1,#2_{#7})}(\overline{#3}#4\overline{#5},\overline{#6}_{#7})\rightarrow}
						}
					}
					{\ifthenelse{\equal{#5}{}}
						{\ifthenelse{\equal{#7}{}}
							{\forall\overline{#6} \diag_{(#1,#2)}(\overline{#3}#4,\overline{#6})}
							{\forall\overline{#6}_{#7} \diag_{(#1,#2_{#7})}(\overline{#3}#4,\overline{#6}_{#7})}
						}
						{\ifthenelse{\equal{#7}{}}
							{\forall\overline{#6} \diag_{(#1,#2)}(\overline{#3}#4\overline{#5},\overline{#6})}
							{\forall\overline{#6}_{#7} \diag_{(#1,#2_{#7})}(\overline{#3}#4\overline{#5},\overline{#6}_{#7})}
						}
					}	
				}
			}
		}
	}

\newcommand*{\hrchy}[2][s]{
	\ifthenelse{\equal{#1}{s}}
		{\Sigma\overset{^{\raisebox{1pt}{\tiny c}}}{_{\raisebox{-1pt}{\tiny #2}}}}
		{\ifthenelse{\equal{#1}{p}}
			{\Pi\overset{^{\raisebox{1pt}{\tiny c}}}{_{\raisebox{-1pt}{\tiny #2}}}}
			{\ifthenelse{\equal{#1}{h}}
				{\mathbf{H}\overset{^{\raisebox{1pt}{\tiny c}}}{_{\raisebox{-1pt}{\tiny #2}}}}
				{bbb}
			}
		}
}
\newcommand*{\hrchyLong}[2][s]{
	\ifthenelse{\equal{#1}{s}}
		{\Sigma\overset{\mkern -15mu ^{\raisebox{1pt}{\tiny c}}}{_{\raisebox{-1pt}{\tiny #2}}}}
		{\ifthenelse{\equal{#1}{p}}
			{\Pi\overset{\mkern -15mu ^{\raisebox{1pt}{\tiny c}}}{_{\raisebox{-1pt}{\tiny #2}}}}
			{\ifthenelse{\equal{#1}{h}}
				{\mathbf{H}\overset{\mkern -15mu ^{\raisebox{1pt}{\tiny c}}}{_{\raisebox{-1pt}{\tiny #2}}}}
				{bbb}
			}
		}
}

\newcommand*{\hrchySeq}[3][s]{
	\ifthenelse{\equal{#1}{s}}
		{\Sigma\overset{^{\raisebox{1pt}{\tiny $ c_{#3} $}}}{_{\raisebox{-1pt}{\tiny #2}}}}
		{\ifthenelse{\equal{#1}{p}}
			{\Pi\overset{^{\raisebox{1pt}{\tiny $ c_{#3} $}}}{_{\raisebox{-1pt}{\tiny #2}}}}
			{\ifthenelse{\equal{#1}{h}}
				{\mathbf{H}\overset{^{\raisebox{1pt}{\tiny $ c_{#3} $}}}{_{\raisebox{-1pt}{\tiny #2}}}}
				{bbb}
			}
		}
}

\DeclareMathOperator{\diag}{Diag}

\begin{document}

\title{Some Model Theoretic Properties of Non-AC Generic Structures}

\author{Ali Valizadeh\hspace*{20pt} Massoud Pourmahdian
  \thanks{Electronic addresses: \texttt{valizadeh.ali@aut.ac.ir,}\\\hspace*{110pt}  \texttt{pourmahd@ipm.ir}}
}
\affil{	
		}


\date{June 03, 2016}
\maketitle

\begin{abstract}

In the context of Hrushovski constructions we take a language $ \mathcal{L} $ with a ternary relation $ R $ and consider the theory of the generic models $ M^{*}_{\alpha}, $ of the class of finite $ \mathcal{L}$-structures equipped with predimension functions $ \delta_{\alpha}, $ for $ \alpha\in(0,1]\cap\mathbb{Q} $. The theory of generic structures of non-AC smooth classes have been investigated from different points of view, including  decidability and their power in interpreting known structures and theories. For a rational $ \alpha\in(0,1], $ first we prove that the theory of $ M^{*}_{\alpha} $ admits a quantifier elimination down to a meaningful class of formulas, called \textit{closure formulas}; and on the other hand we prove that $ \theory(M^{*}_{\alpha}) $ does not have the finite model property.

\end{abstract}

\section{Introduction and preliminaries}\label{secIntroPrlmn}
Our main contribution in this paper is to further investigate the model theory of non-algebraic generic structures. In a language $ \mathcal{L} $ with a ternary relation, for each $ \alpha\in(0,1] $ one can associate to the class of finite $ \mathcal{L} $-structures some certain function $ \delta_{\alpha}, $ called a \textit{predimension} (\Cref{dfnPredimension}). In the present work we mainly focus on understanding particular \textit{generic} structures (\Cref{dfnGenericModel}), $ M^{*}_{\alpha} $ which were introduced by Hrushovski in \cite{Hrushovski-SimpLascarGroups} and were later studied by various authors in \cite{Pourmahd-SmoothClasses}, \cite{Evans&Wong-SomeRemarksonGen}, and \cite{Brody&Laskowski-OnRationalLimits}. For a rational $ \alpha, $ it has been proved by Evans and Wong that all finite graphs are interpretable in $ M_{\alpha}^{*}, $ and as a consequence, $ \theory(M_{\alpha}^{*}) $ has the strict order property and is undecidable. Later, Brody and Laskowski gave another source of undecidability for $ M_{\alpha}^{*} $ by interpreting Robinson arithmetic in a subtheory of $ \theory(M_{\alpha}^{*}). $ These results suggest that $ \theory(M_{\alpha}^{*}) $ should be assumed to be quite ``wild".

However, the starting point for investigating theory of generic structures is to determine whether the type of tuples can be fully determined in terms of their corresponding closures. This happens in particular within the theory of ``tame" generic structures, where the type of a tuple is given completely by (Boolean combinations of) certain $ \Sigma_{1} $-formulas which indicate the existence of certain diagrams in the closure of that specific tuple.

Our first result suggests that, despite the fact that the theory of $ M_{\alpha}^{*} $ behaves in a wild way, it still shows some levels of tameness by permitting the type of tuples to be fully described by their closures; although this description is obtained through more complex formulas, called \textit{closure formulas}. In addition, this result provides the means to further distinguish the theory of a certain ultraproduct of generic structures, hence strengthening the result by Brody and Laskowski (Theorem 4.9 of \cite{Brody&Laskowski-OnRationalLimits}).

More on the wildness side, we further refine the techniques developed in \cite{Brody&Laskowski-OnRationalLimits} and show that the structure $ \langle\mathbb{Q},<\rangle $ is interpretable in $ M_{\alpha}^{*}. $ Hence we prove that $ \theory(M_{\alpha}^{*}) $ does not have the finite model property. This answers a question posed by Evans and Wong in \cite{Evans&Wong-SomeRemarksonGen}.

\begin{Large}
\end{Large}

\paragraph*{Setting.} Throughout we will be working in a relational language $ \mathcal{L}, $ with only one ternary relation $ R. $ Finite $ \mathcal{L} $-structures are denoted by $ A,B,C,\ldots  $ and by $ M,N,\ldots $ we mean some arbitrary $ \mathcal{L} $-structures. By $ A\subset_{\omega} M $ we mean that $ A $ is a finite substructure of $ M. $ By $ \mathcal{C} $ we denote a class of finite $ \mathcal{L} $-structures, always presumed to contain $ \emptyset. $ For a structure $ M, $ by $ \age(M)$ we denote the class of its finite substructures. For a given $ \mathcal{C}, $ by $ \bar{\mathcal{C}} $ we denote the class of all $ \mathcal{L} $-structures $ M $ with $ \age(M)\subseteq\mathcal{C}. $ Finally for $ A,B\subseteq C, $ the structure induced by $ C $ on $ A\cup B $ is denoted by $ AB. $

We begin by recalling some of the basic definitions in the context of \textit{smooth classes}, and for a thorough investigation of smooth classes we refer the reader to \cite{Wagner-Relational}, \cite{Baldwin&Shi-StableGen}, \cite{Kueker&Laskowski-GenericStructures}, and \cite{Baldwin-FieldGuide}. Unless otherwise stated, by $ (\mathcal{C},\sqsubseteq) $ we mean an arbitrary smooth class and all the structures are assumed to be an element of $ \mathcal{C}. $

\begin{dfn}\label{dfnFreeJoin}
For every $ A\subseteq B,C $ with $ B\cap C = A, $\\
(i) we say that $ B $ is \textit{free} over $ A $ with respect to (or from) $ C, $ if there is no relation with at least one node in $ B\backslash A $ and at least one node in $ C\backslash A. $\\
(ii) The structure $ D $ is called the \textit{free join} of $ B $ and $ C $ over $ A $, denoted by $ \freeJoin{B}{A}{C}, $ if the universe of $ D $ is $ B\cup C $ and we have:
\[ R^{D} = R^{B}\cup R^{C}. \]
\end{dfn}
\begin{dfn}\label{dfnIntExt} Suppose that $ A\subseteq B, $\\
(i) we say that $ B $ is a $ \sqsubseteq $-\textit{intrinsic extension} of $ A, $ if for any $ C $ with $ A\subseteq C\subsetneq B $ we have that $ C\not{\sqsubseteq}B. $ In this situation we write $ \intCl[\sqsubseteq]{A}{B}. $\\
(ii) We say that $ (A,B) $ is a \textit{minimal pair}, in notation $ \minPair[\sqsubseteq]{A}{B,}$ if for any $ C $ with $ A\subseteq C\subsetneq B $ we have $ A\sqsubseteq C $ but $ A\not{\sqsubseteq} B. $
\end{dfn}

\begin{dfn}\label{dfnOldCl}
Suppose that $ \intCl[\sqsubseteq]{AB}{C} $. We say that $ C $ is $ \sqsubseteq $-\textit{old} over $ A $ with respect to $ B $ whenever $ \intCl[\sqsubseteq]{A}{C} $ and $ C $ is free over $ A $ with respect to $ B. $ 
Otherwise we say that $ C $ is $ \sqsubseteq $-\textit{new} over $ A $ with respect to $ B. $
\end{dfn}
It follows from the definition above that whenever $ \intCl[\sqsubseteq]{AB}{C} $ and $ C $ is free over $ A $ with respect to $ B, $ then $ \intCl[\sqsubseteq]{A}{C} $. The following lemma is standard and follows easily from the definitions above:
\begin{lma}\label{lmaIntClUnion} Let $ M\in\bar{\mathcal{C}}. $ The following hold in $ M $:\\
(i) If $ \intCl[\sqsubseteq]{A}{B} $ and $ \intCl[\sqsubseteq]{B}{C} $, then $ \intCl[\sqsubseteq]{A}{C}.$\\
(ii) If $ \intCl[\sqsubseteq]{A}{B_{1}} $ and $ \intCl[\sqsubseteq]{A}{B_{2}} $, then $ \intCl[\sqsubseteq]{A}{B_{1}B_{2}}. $\\
(iii) If $ \intCl[\sqsubseteq]{A}{B_{1}} $ and $ \intCl[\sqsubseteq]{A}{B_{2}} $, then $ \intCl[\sqsubseteq]{B_{1}}{B_{1}B_{2}}. $\\
(iv) If $ \minPair[\sqsubseteq]{A}{B}, $ then $ \intCl[\sqsubseteq]{A}{B}. $\\
(v) If $ \intCl[\sqsubseteq]{A}{B}, $ then there is a finite sequence $ \minPair[\sqsubseteq]{B_{0}}{B_{1}}\minPair[\sqsubseteq]{}{}\cdots\minPair[\sqsubseteq]{}{B_{n}} $ such that $ B_{0}=A $ and $ B_{n}=B. $\qed
\end{lma}

\begin{dfn}\label{dfnGenericModel}
A countable $ \mathcal{L} $-structure $ M $ is called $ (\mathcal{C},\sqsubseteq) $\textit{-generic} if the following conditions hold:\\
(1) $ M\in\bar{\mathcal{C}}, $\\
(2) whenever $ A\sqsubseteq M $ and $ A\sqsubseteq B\in\mathcal{C} $, there is $ B'\subseteq M $ such that $ B'\cong_{A}B $ and $ B'\sqsubseteq M. $\\
(3) $ M $ is the union of a chain $ \{A_{i}:i\in\omega\} $ where $ A_{i}\in\mathcal{C} $ and $ A_{i}\sqsubseteq A_{i+1}, $ for each $ i\in\omega. $
\end{dfn}

\begin{dfn}\label{dfnChi}
Let $ M\in\bar{\mathcal{C}}, $ $ A\subseteq M, $ and $ A\subseteq B\in\mathcal{C}. $ Then:\\
(i) By a \textit{copy} of $ B $ over $ A $ in $ M $, we mean the image of an embedding of $ B $ over $ A $ into $ M. $\\
(ii) By $ \chi_{M}(B/A) $ we mean the number of distinct copies of $ B $ over $ A $ in $ M $ (which can be infinite).
\end{dfn}

\begin{dfn}\label{dfnClosure}
For any $ M\in\bar{\mathcal{C}} $ and $ A\subset_{\omega} M $ we define
\begin{align*}
&\cl_{M}(A):=\bigcup\{B\subset_{\omega} M|\intCl[\sqsubseteq]{A}{B} \},\\
&\cl^{n}_{M}(A):=\bigcup\{B\subset_{\omega} M|\intCl[\sqsubseteq]{A}{B}, \card[B\backslash A]< n \}.
\end{align*}
\end{dfn}
\begin{factB}\label{factClosureProperties}
(i) $ \cl_{M}(A) $ is the unique minimal structure containing $ A $ closed in $ M. $\\
(ii) If $ M $ is $ (\mathcal{C},\sqsubseteq) $-generic, then the closure of each finite substructure of $ M $ is finite.\qed
\end{factB}
\begin{dfn}\label{dfnAC}
A smooth class $ \mathcal{C} $ is said to have the \textit{algebraic closure property}, in short \textit{AC}, if there is a function $ \mu:\omega\times\omega\rightarrow\omega $ such that for any $ A,B\in\mathcal{C} $ with $ A\sqsubseteq_{i}B $ and any $ M\in\bar{\mathcal{C}} $ we have $ \chi_{M}(B/A)<\mu(\card[A],\card[B]). $
\end{dfn}
If a class $ \mathcal{C} $ has AC, then for any $ M\in\bar{\mathcal{C}},\hspace*{2pt} n\in\omega, $ and $ A\subset_{\omega} M $, $ \cl^{n}_{M}(A) $ is finite.

The notion of genericity for a smooth class $ (\mathcal{C},\sqsubseteq) $ is not necessarily first order expressible, but there exists a first order analogue, which is called \textit{semigenericity} first introduced in \cite{Baldwin&Shelah-Randomness} for smooth classes with AC. It was then generalized to a broader context by the second author in \cite{Pourmahd-SimpleGen} and \cite{Pourmahd-SmoothClasses}.
\begin{dfn}\label{dfnSemigenericModel}
An $ \mathcal{L} $-structure $ M $ is called $ (\mathcal{C},\sqsubseteq) $\textit{-semigeneric} if:\\
(1) $ M\in\bar{\mathcal{C}}, $\\
(2) for any $ A\subset_{\omega} M $, any $ B $ with $ A\sqsubseteq B\in\mathcal{C} $, embedding $ g:A\rightarrow M $, and $ n\in\omega $, there is an embedding $ \bar{g}:B\rightarrow M $ that extends $ g $ and
\[ \cl^{n}_{M}(\bar{g}B) = \freeJoin{\cl^{n}_{M}(gA)}{gA}{\bar{g}B},  \]
\end{dfn}
that is, \oldClLiteral{\cl^{n}_{M}(gA)}{gA}{\bar{g}B}.

The reader is referred to \cite{Baldwin&Shelah-Randomness}, to see that the notion of semigenericity can be expressed by a first order theory.

A vast variety of smooth classes is obtained from the class of finite structures equipped with certain functions called \textit{predimension}s.

\begin{dfn}\label{dfnPredimension}
A function $ \delta:\mathcal{C}\rightarrow \mathbb{R}^{\geq 0} $ is called a \textit{predimension} if for any $ A,B\in\mathcal{C} $ we have\\
(1) $ \delta(\emptyset)=0, $\\
(2) if $ A\cong B $ then $ \delta(A)=\delta(B), $\\
(3) $ \delta(A\cup B)\leq\delta(A)+\delta(B)-\delta(A\cap B). $
\end{dfn}

A given class $ \mathcal{C} $ can be equipped with several predimension functions, but we will be interested in ones which are proved to have interesting model theoretic properties. For a real number $ \alpha\in (0,1] $ define $ \delta_{\alpha}:\mathcal{C}\rightarrow\mathbb{R}^{\geq 0} $ as
\[ \delta_{\alpha}(A)=\card[A]-\alpha\card[R^{A}], \]
where by $ \card[R^{A}] $ we mean the number of relations in $ A $. We assume that $ R $ is symmetric and realized only with distinct triples of elements. Based on these predimensions one can define two \textit{strong} submodel relations on $ \mathcal{C} $ that are denoted by $ \leq_{\alpha} $ and $ \leqst{}{}_{\alpha}. $
\begin{dfn}\label{dfnSmoothClassExamples} For $ A,B\in\mathcal{C} $ we write\\
(i) $ A\leq_{\alpha} B $ if for any $ C\in\mathcal{C} $ with $ A\subseteq C\subseteq B $ we have $ \delta_{\alpha}(A)\leq\delta_{\alpha}(C). $\\
(ii) $ \leqst[\leq_{\alpha}]{A}{B} $ if for any $ C\in\mathcal{C} $ with $ A\subsetneq C\subseteq B $ we have $ \delta_{\alpha}(A)<\delta_{\alpha}(C). $\\
(iii) Also let $ \classPositEq{\alpha}:=\{ A\in\mathcal{C} \hspace*{3pt}|\hspace*{3pt} \emptyset\leq_{\alpha} A\} $ and $ \classPosit{\alpha}:= \{ A\in\mathcal{C} \hspace*{3pt}|\hspace*{3pt} \leqst[\leq_{\alpha}]{\emptyset}{A}\}. $\\
If $ A\leq_{\alpha} B $ (or $ \leqst[\leq_{\alpha}]{A}{B}, $ ) we say that $ A $ is $ \leq_{\alpha} $-strong (or $ \leq^{*}_{\alpha} $-strong) in B.
\end{dfn}

The corresponding intrinsic extensions for the classes $ \classPositEq{\alpha} $ and $ \classPosit{\alpha} $ are denoted respectively by $ \leq_{i,\alpha} $ and $ \leq^{*}_{i,\alpha}. $ In these classes the notion of intrinsic extension can be described concretely as follows.
\begin{factB}\label{factIntrisicForPredim}
(i) $ A\leq_{i_{\alpha}}B $ if and only if for any $ C $ with $ A\subseteq C\subsetneq B $ we have $ \delta_{\alpha}(B)<\delta_{\alpha}(C). $\\
(ii) $ A\leq^{*}_{i_{\alpha}}B $ if and only if for any $ C $ with $ A\subseteq C\subsetneq B $ we have $ \delta_{\alpha}(B)\leq\delta_{\alpha}(C). $\qed
\end{factB}

Both $ \classPositEqOpr{\alpha} $ and $ \classPositOpr{\alpha} $ are smooth classes with the full amalgamation property possessing the properties necessary for the existence of a generic structure. Hence we reserve $ M_{\alpha} $ and $ M^{*}_{\alpha} $ respectively for the generic structures of these two classes. Also set $ T_{\alpha}=\theory(M_{\alpha}) $ and $ T^{*}_{\alpha}=\theory(M_{\alpha}^{*}).$ Denote by $ T_{\alpha\mhyphen \sem} $ and $ T^{*}_{\alpha\mhyphen \sem} $ respectively the theory of the semigeneric structures of $ \classPositEqOpr{\alpha} $ and $ {\classPositOpr{\alpha}}. $ Notice that since full amalgamation holds for these two classes, their corresponding generics satisfy respectively $ T_{\alpha\mhyphen \sem} $ and $ T^{*}_{\alpha\mhyphen \sem}. $ Also we let $ \alpha\mhyphen \cl(A) $ and $ \alpha\mhyphen \cl^{*}(A) $ denote the corresponding closures.

It is worth noting that if $ \alpha $ is an irrational, then the classes $ \classPositEq{\alpha} $ and $ \classPosit{\alpha} $ are the same and for any finite $ A,B, $ the relation $ A\leq_{\alpha}B $ is equivalent to $ \leqst[\leq_{\alpha}]{A}{B}. $

Moreover by Lemma 3.19 and 3.22 of \cite{Baldwin&Shi-StableGen}, we know that for every $ \alpha\in(0,1], $ the class $ \classPositEqOpr{\alpha} $ has AC; the reason is that whenever $ A\leq_{i,\alpha} B, $ the relative predimension of $ B $ over $ A $ is negative and each additional copy of $ B $ strictly decrease the predimension and hence for any $ M\in\bar{\mathcal{C}}, $ the number $ \chi_{M}(B/A) $ is bounded above. In addition, we refer to the proof of Lemma 3.19 and 3.22 in \cite{Baldwin&Shi-StableGen}, for the reason why this upper bound is uniformly determined by $ \card[A], \card[B] $ and $ \alpha. $

On the other hand the class $ \classPositOpr{\alpha}, $ for rational $ \alpha $, does not have AC, the reason is that for a structure $ A\in\classPosit{\alpha}, $ there are intrinsic extensions $ B $ with $ \delta_{\alpha}(B/A)=0. $ This fact allows the copies of $ B $ over $ A $ to be infinite. But even in this class, $ \chi_{M}(B/A) $ is uniformly bounded for intrinsic extensions with $ \delta_{\alpha}(B/A)<0. $ In such cases being $ \leq^{*}_{i,\alpha} $-intrinsic extension implies $ \leq_{i,\alpha} $-intrinsic extension and we have the following fact, which is the same as Lemma 3.19 in \cite{Baldwin&Shi-StableGen} but it is rephrased in a way which is more suitable for the present context:

\begin{fact}\label{factUpperBound}
For every $ \alpha\in(0,1], $ there is a function $ \mu_{\alpha}:\omega\times\omega\longrightarrow\omega, $ such that for every $ A\leq_{i,\alpha} B\in\classPositEq{\alpha}, $ and every $ M\in\classPositEq[\bar{C}]{\alpha}, $ we have $ \chi_{M}(B/A)<\mu_{\alpha}(\card[A],\card[B]). $ Moreover
\[ \mu_{\alpha}(\card[A],\card[B])\leq \frac{\delta_{\alpha}(A)}{\alpha\card[R^{B/A}]},\]
where by $ R^{B/A} $ we mean the set of relations with at least one node in $ B\backslash A $ and at least one node in $ A. $\qed
\end{fact}

%

In order to draw a dividing line in terms of the classes $ \classPositEqOpr{\alpha} $ and $ \classPositOpr{\alpha} $ that are the typical examples of smooth classes with and without AC, we collect some results about these classes.  An important feature here is that the types of tuples in the models of the theory of generic structures are expected to be determined by their closures. Hence one can see that in the presence of AC, that bounds the size of the closures, the smooth class leads to a stable context.
\begin{factB}\label{factStableContext}
\hspace{15pt}(i) For every $ \alpha\in(0,1], $ we have that $ T_{\alpha}=T_{\alpha\mhyphen \sem}. $ Moreover $ T_{\alpha} $ is a near-model-complete but not model-complete theory(\cite{Baldwin&Shelah-Randomness}).\\
\hspace{15pt}(ii) If $ \alpha\in\mathbb{Q}, $ then $ T_{\alpha} $ is $ \omega $-stable(\cite{Baldwin&Shi-StableGen}).\\
\hspace{15pt}(iii) If $ \alpha\not\in\mathbb{Q}, $ then $ T_{\alpha} $ is strictly stable (\cite{Baldwin&Shi-StableGen}). Furthermore $ T_{\alpha} $ is equal to $ T^{\alpha}, $ the almost-sure theory of random graphs with edge probability $ n^{-\alpha}. $ Hence it has the finite model property(\cite{Baldwin&Shelah-Randomness}).\qed
\end{factB}

But for $ \classPositOpr{\alpha}, $ for $ \alpha $ a rational, the situation for the theory of generic changes dramatically:
\begin{factB}\label{factUnstableContext}
 \hspace{25pt}(i) $ T^{*}_{\alpha\mhyphen \sem} $ interprets Robinson arithmetic and is essentially undecidable(\cite{Brody&Laskowski-OnRationalLimits}). \\
\hspace{25pt}(ii) All finite graphs are interpretable in $ M^{*}_{\alpha} $ and hence $ T^{*}_{\alpha} $ is undecidable and has the strict order property(\cite{Evans&Wong-SomeRemarksonGen}).\qed
\end{factB}

Also for a smooth class $ \mathcal{C} $ one can associate a theory $ T_{\text{Nat}} $ in an expanded language $ \mathcal{L}_{+}\supseteq\mathcal{L} $ which reduces the notion of strong submodel in $ \mathcal{L} $ to the notion of $ \mathcal{L}_{+} $-substructure. Precise definitions and more investigations can be found in \cite{Pourmahd-SmoothClasses} and \cite{Pourmahd-SimpleGen}.

It is also shown in \cite{Pourmahd-SmoothClasses} and \cite{Pourmahd-SimpleGen} that the class of existentially closed models of $ T^{*}_{\alpha\mhyphen \text{Nat}} $ is rather well-behaved. In particular it is shown that this class is simple in the sense of \cite{Pillay-ForkingExistential}.

\section{Closure formulas}\label{secQE}
Recall that if $ \alpha  $ is irrational, then the classes $ \classPositEqOpr{\alpha} $ and $ \classPositOpr{\alpha} $ are the same, and moreover the theory of their generic structure is near-model-complete. For rational $ \alpha $, near-model-completeness still holds for $ \classPositEqOpr{\alpha} $-generic $ M_{\alpha}. $ In this section we prove that for a rational $ \alpha\in(0,1], $ the theory of the $ \classPositOpr{\alpha} $-generic $ M^{*}_{\alpha}, $ admits quantifier elimination down to a certain class of formulas, called \textit{closure formulas.} Also as an application, we investigate the theory of a certain ultraproduct of generic structures.

For the rest of this section we fix a rational $ \alpha $ and to ease the notation, we drop the subscript $ \alpha $ from $ \leq_{\alpha}, \leq^{*}_{\alpha}, $ and $ \leq^{*}_{i,\alpha}, $ and will respectively write $ \leq, \leq^{*}, $ and $ \leq^{*}_{i} $ instead.

Closure formulas (\ref{dfnClosureFormula}) are roughly speaking $ \mathcal{L} $-formulas $ \phi(\bar{x}) $ in which all the quantifiers range through the intrinsic extensions of $ \bar{x}. $ The closure formulas are inductively built through the following hierarchy. Let:
\begin{align*}
& X:=\Big\{x_{i}=x_{i}, x_{i}\neq x_{i} \Big| i\in\omega\Big\},\\
& S^{c}_{0}:=\Big\{\diagI{x}{}\Big| A\in\classPosit{1}\Big\}\cup X, \\
& P^{c}_{0}:=\Big\{\neg\diagI{x}{}\Big| A\in\classPosit{1}\Big\}\cup X,
\end{align*}
and define $ \hrchy{0} $ and $ \hrchy[p]{0} $ to be the sets which are obtained respectively from $ S^{c}_{0} $ and $ P^{c}_{0} $ by closure under conjunctions and disjunctions.

For $ {\intCl{A}{B}\in\classPosit{},} $ denote by $ \diagCl{B}{\bar{x}}{\bar{y}} $ the formula that expresses the diagram of $ B $ over $ A $. We make the convention that $ \bar{x} $ are variables associated to the elements of $ A $ and $ \bar{y}$ are variables containing $ B\backslash A $ and not intersecting $ \bar{x}. $ We also allow $ \bar{y} $ to contain dummy variables.

For any $ n\geq 1 $ define $ S^{c}_{n} $ to be the following set of formulas:
\begin{align}\label{eqnScn}
\Big\{ \diagClIExTupLParenth{B}{x}{ }{ }{y}{ }{B} \bigwedge_{i=1}^{m}\theta_{i}(\bar{x}\bar{y})\Big)\Big| \intCl{A}{AB}\in\classPosit{1}, \theta_{i}(\bar{x}\bar{y})\in P^{c}_{n-1} \Big\},
\end{align}
and similarly $ P^{c}_{n} $ to be the set:
\begin{align}\label{eqnPcn}
\Big\{\diagClIAllTupLParenth{B}{x}{}{}{y}{}{B}\bigvee_{i=1}^{m}\theta_{i}(\bar{x}\bar{y})\Big)\Big| \intCl{A}{AB}\in\classPosit{1}, \theta_{i}(\bar{x}\bar{y})\in S^{c}_{n-1} \Big\}.
\end{align}
Also define $ \hrchy{n} $ and $ \hrchy[p]{n} $ respectively as closures under positive Boolean operations of $ S^{c}_{n} $ and $ P^{c}_{n}. $ Finally consider the following sets:
\begin{align*}
& S^{c}:=\bigcup_{n\in\omega} S^{c}_{n},\quad P^{c}:=\bigcup_{n\in\omega}P^{c}_{n},\quad \hrchy[h]{n}:=\hrchy{n}\cup\hrchy[p]{n},\\
& \hrchy{}:=\bigcup_{n\in\omega}\hrchy{n},\quad \hrchy[p]{}:=\bigcup_{n\in\omega}\hrchy[p]{n},\quad\text{and}\quad \hrchy[h]{}:=\bigcup_{n\in\omega}\hrchy[h]{n}.
\end{align*}
\begin{dfn}\label{dfnClosureFormula}
By a \textit{closure formula} we mean a formula in $ \hrchy[h]{}. $ A closure formula is called \textit{basic}, if it lies in $ S^{c}\cup P^{c}. $
\end{dfn}


To avoid confusion, for every basic closure formula $ \phi(\bar{x}) $ sometimes we have added $ (A,B) $ as subscript to $ \phi(\bar{x}), $ hence we display it by $ \phi_{(A,B)}(\bar{x}), $ to emphasise on the initial diagrams appearing in $ \phi(\bar{x}). $

For the proof of \Cref{thmQE}, we need to work with the diagrams that appear in basic closure formulas. To each basic closure formula $ \phi_{(A,B)}(\bar{x})\in S^{c}\cup P^{c}, $ we associate a tree $ \tau_{\phi_{(A,B)}}, $ as in \Cref{dfnTree}. For such a formula $ \phi_{(A,B)}(\bar{x}), $ its tree $ \tau_{\phi_{(A,B)}} $ is comprised of a pair $ \langle V_{\phi_{(A,B)}},E_{\phi_{(A,B)}}\rangle $ and is defined inductively as follows.
\begin{dfn}\label{dfnTree}
For $ \phi_{A}(\bar{x})\in S^{c}_{0}\cup P^{c}_{0}, $ define $ \tau_{\phi_{A}} =\langle \{A\},\emptyset\rangle.$ For $ n\geq 1, $ if the formula $ \phi_{(A,B)}(\bar{x})\in S^{c}_{n}\cup P^{c}_{n} $ is of the form of the elements in (\ref{eqnScn}) or (\ref{eqnPcn}) define
\begin{align*}
&V_{\phi_{(A,B)}}=\{A,B\}\cup\bigcup_{i=1}^{m}V_{\theta_{(B,C_{i})}},\\
&E_{\phi_{(A,B)}}=\{(A,B)\}\cup\bigcup_{i=1}^{m}\{(B,C_{i})\}\cup E_{\theta_{(B,C_{i})}}.
\end{align*}
We call $ A $ the root of $ \tau_{\phi_{(A,B)}} $ and denote by $ \lambda_{\phi} $ the number of the leaves of $ \tau_{\phi_{(A,B)}}. $
\end{dfn}

One can easily check that there is a one-to-one correspondence between trees and formulas. The next remark and lemmas demonstrate some properties of closure formulas and propose a more canonical way to describe them. They also provide tools needed for proving \Cref{thmQE}.

\begin{rmrkB}\label{rmrkBasicFactsAboutClosureFormula}
(i) $ \hrchy{0}=\hrchy[p]{0}, S^{c}_{n}\subseteq S^{c}_{n+1}, \hrchy{n}\subseteq\hrchyLong{n+1},  \hrchy[p]{n}\subseteq\hrchyLong[p]{n+1}, $ for any $ n\in\omega. $ Moreover $ \hrchy[p]{}\subseteq\hrchy{}. $\\
(ii) Every $ S^{c}_{1} $-formula is equivalent either to an antilogy or to a formula of the form $ \diagClIExTupLParenth{B}{x}{}{}{y}{}{}, $ for some $ \intCl{A}{B}\in\classPosit{}. $\\
(iii) Every $ P^{c}_{1} $-formula is equivalent to a tautology, an antilogy, or to a formula of the form $ \diagClIAllTupLParenth{B}{x}{}{}{y}{}{neg}, $ for some $ \intCl{A}{B}\in\classPosit{}. $
\end{rmrkB}

\begin{lmaB}\label{lmaHierarchyCanonicalFormSigma}
(i) If $ \frmI{x}{1}\in P^{c}_{m}$ and $ \frmI{x}{2}\in P^{c}_{m'}, $ then $ \frmI{x}{1}\wedge\frmI{x}{2} $ is logically equivalent to a formula of the following form:
\[ \Big(\frmI{x}{1}\wedge\frmI[\phi']{x}{2}\Big)\vee\Big(\frmI{x}{2}\wedge\frmI[\phi']{x}{1}\Big)\vee\bigvee_{i=1}^{k}\frmI[\chi]{x}{i}, \]
where $ \frmI[\phi']{x}{1},\frmI[\phi']{x}{2}\in P^{c}_{1} $ and for each $ i, $ $ \frmI[\chi]{x}{i}\in P^{c}_{n}, $ where $ n=\max\{m,m'\}. $ Hence a conjunction of $ \hrchy[p]{} $-formulas is equivalent to a disjunction of $ P^{c} $-formulas.
\\
(ii) If $ \frmI{x}{1}\in S^{c}_{m} $ and $\frmI{x}{2}\in S^{c}_{m'},$ then $ \frmI{x}{1}\wedge\frmI{x}{2} $ is equivalent to a disjunction of $ S^{c}_{n} $-formulas, where $ n=\max\{m,m'\}. $\\
(iii) Let $ \phi(\bar{x}\bar{y}\bar{z}) $ be a quantifier free $ \mathcal{L} $-formula implying $ \diagCl[A]{B}{\bar{x}}{\bar{y}}, $ for some $ \intCl{A}{B}\in\classPosit{}. $ Then for any $ \psi_{1}(\bar{x}\bar{y}\bar{z}),\ldots,\psi_{m}(\bar{x}\bar{y}\bar{z})\in S^{c}_{n}, $ the following formula, $ \theta(\bar{x}\bar{y}\bar{z}) $, is a $ \hrchyLong[p]{n+1} $-formula:
\[ \forall\bar{y}\Big(\phi(\bar{x}\bar{y}\bar{z})\rightarrow\bigvee_{i=1}^{m}\psi_{i}(\bar{x}\bar{y}\bar{z})\Big). \]
\end{lmaB}
\begin{proof} If $ n=\max\{m,m'\}, $ based on \Cref{rmrkBasicFactsAboutClosureFormula} we can assume that $ \frmI{x}{1}, \frmI{x}{2}\in P^{c}_{n}. $ So suppose that $ \frmI{x}{1} $ and $ \frmI{x}{2} $ are respectively as followings:
\begin{align*}
&\diagClIAllTupLParenth[A_{1}]{C}{x}{}{}{z}{1}{B}\bigvee_{j=1}^{m_{1}}\theta^{1}_{j}(\bar{x}\bar{z})\Big),\\
&\diagClIAllTupLParenth[A_{2}]{C}{x}{}{}{z}{2}{B}\bigvee_{j=1}^{m_{2}}\theta^{2}_{j}(\bar{x}\bar{z})\Big).
\end{align*}
If $ A_{1}\not\cong A_{2} $, it is easy to check that $ \frmI{x}{1}\wedge\frmI{x}{2} $ is logically equivalent to
$ \frmI{x}{1}\vee\frmI{x}{2}. $ So suppose that $ A_{1}\cong A_{2} $ and name this common structure $ A. $ Then define $ \phi'_{1}(\bar{x}) $ and $ \phi'_{2}(\bar{x}) $ respectively to be $ \diagClIAllTupLParenth{C}{x}{}{}{z}{1}{neg} $ and $ \diagClIAllTupLParenth{C}{x}{}{}{z}{2}{neg}. $ Note that for a tuple $ \bar{a} $ there are actually four possible cases; the first one is when neither $ C_{1} $ nor $ C_{2} $ are satisfied over $ \bar{a}. $ The second case is when $ C_{1} $ is satisfied over $ \bar{a} $ but $ C_{2} $ is not. The third case is when $ C_{2} $ is satisfied over $ \bar{a} $ but $ C_{1} $ is not. And the last case, when $ C_{1} $ and $ C_{2} $ are both satisfied over $ \bar{a}. $ It can be easily seen that in the first three cases, $ \frmI{x}{1}\wedge\frmI{x}{2} $ is respectively equivalent to $ \phi'_{1}(\bar{x}) \wedge\phi'_{2}(\bar{x}), \frmI{x}{1}\wedge\phi'_{2}(\bar{x}) $ or $ \frmI{x}{2}\wedge\phi'_{1}(\bar{x}). $ Now in the last case, note that by \Cref{lmaIntClUnion} we have $ \intCl{A}{C_{1}C_{2}}. $ Notice that there is a finite number of possible diagrams $ C'_{1},\ldots,C'_{k} $ for $ C_{1}\cup C_{2} $ to be realized over $ A ,$ and therefore the conjunction $ \frmI{x}{1}\wedge\frmI{x}{2} $ is logically equivalent to:
\begin{align*}
\bigvee_{i=1}^{k}\Big[\diagClIAllTupLParenth{C'}{x}{}{}{z}{i}{B}\bigdoublewedge_{\substack{j=1,l=1}}^{\hspace*{2pt}m_{1},m_{2}}(\theta^{1}_{j}(\bar{x}\bar{z}_{i})\vee\theta^{2}_{l}(\bar{x}\bar{z}_{i}))\Big)\Big],
\end{align*}
which introduces the required $ \frmI[\chi]{x}{i} $s.

(ii) By essentially a similar argument to (i).

(iii) First note that the formula $ \theta(\bar{x}\bar{y}\bar{z}) $ is equivalent to:
\begin{align}\label{eqnTemp1}
\diagClIAllTupLParenth{B}{x}{}{}{y}{}{B}\neg\phi(\bar{x}\bar{y}\bar{z})\vee\bigvee_{i=1}^{m}\psi_{i}(\bar{x}\bar{y}\bar{z})\Big),
\end{align}
and for a set $ C $ with $ \card[C]=\card[\bar{z}], $ there exist finitely many complete diagrams which have $ B\cup C $ as their universe and extend $ B $ as their substructure. Enumerate them by $ BC_{1},\ldots,BC_{l} $ and notice that since the language is finite relational, \Cref{eqnTemp1} is equivalent to:
\[ \diagClIAllTupLParenth{B}{x}{}{}{y}{}{B}\bigvee_{k=1}^{l}Diag_{ABC_{k}}(\bar{x}\bar{y}\bar{z})\vee\bigvee_{i=1}^{m}\psi_{i}(\bar{x}\bar{y}\bar{z})\Big). \]

Therefore by part (i) of \Cref{rmrkBasicFactsAboutClosureFormula}, this formula is a $ \hrchyLong[p]{n+1} $-formula.
\end{proof}

\begin{crl}\label{crlCanonicalForm}
Every closure formula $ \phi(\bar{x}\bar{y}) $ is equivalent to a disjunction of $ S^{c} $-formulas of the following form
\begin{equation}\label{eqnSigmaWLG}
\diagClIExTupLParenth[AB]{C}{}{\bar{x}\bar{y}}{}{z}{}{B}\bigwedge_{i=1}^{m}\theta_{i}(\bar{x}\bar{y}\bar{z})\wedge\theta(\bar{x}\bar{y}\bar{z})\Big),
\end{equation}
where $ \theta(\bar{x}\bar{y}\bar{z})\in P^{c}_{n-1}, $ and for each $ 1\leq i\leq m, $ we have $ \theta_{i}(\bar{x}\bar{y}\bar{z})\in P^{c}_{1}. $ \qed
\end{crl}

Notice that in \cite{Baldwin&Shelah-Randomness}, in fact a ``combinatorial measurement" was derived for every $ \alpha\in(0,1] $ in such a way that it decides the satisfaction of a given formula $ \phi(\bar{x}) $ uniformly in all models of $ T_{\alpha\mhyphen \sem}. $ More precisely, since AC holds in $ \classPositEqOpr{\alpha} $, there exists a natural number $ l_{\phi} $ such that for any $ M\models  T_{\alpha\mhyphen \sem}, $ and any $ \bar{a}\in M, $ the satisfaction of $ \phi(\bar{a}) $ depends only on the isomorphism type of $ cl^{l_{\phi}}_{M}(\bar{a}). $ Again $ cl^{l_{\phi}}_{M}(\bar{a}), $ thanks to AC, is finite (uniformly in all models of $ T_{\alpha\mhyphen \sem} $ ) and hence $ T_{\alpha\mhyphen \sem} $ is both complete and near-model-complete.

On another basis, in our case we propose a measurement in terms of ``logical complexity" which relates the satisfaction of $ \phi(\bar{x}) $ to a formula which describes a bounded part of the closure of $ \bar{x}. $ To introduce this measurement a machinery needs to have appeared before, that works even in the absence of AC. More precisely, for any $ \phi(\bar{x}) $ one can find a natural number $ n=n_{\phi} $ and a closure formula in $ \hrchy[h]{n} $ that is equivalent to $ \phi(\bar{x}) $ under $ T^{*}. $

The next definition, distinguishes the non-AC from the AC part inside $ \classPositOpr{}. $ In this definition we highlight a distinction among closure formulas, and introduce a dividing line one side of which contains the formulas with a finite nature as in the classes with AC. In fact this \textit{non-primary} closure formulas are the ones constructed from $ \leq $-intrinsic extensions. But the other side of the dividing line contains the \textit{primary} formulas that in the classes without AC, constitute the noteworthy part of closure formulas. It can be subsequentially seen in the proof of \Cref{lmaPrimaryFormulas} how the non-primary formulas collapse to the lower levels of the hierarchy, namely to $ \hrchy[s]{1} $-formulas. The definition of primary formulas is done by induction on the levels of $ \hrchy[h]{}. $
\begin{dfn}\label{dfnPrimaryClosureFormula} We define \textit{primary} formulas as follows:\\
(i) Both $ P^{c}_{0} $ and $ P^{c}_{1} $-formulas are \textit{primary}. For $ n\geq 2, $ the formula $ \phi_{(A,B)}(\bar{x})\in P^{c}_{n} $ (which is of the form of the elements of (\ref{eqnPcn})) is primary if $ A\leq B $ and for each $ 1\leq i\leq m $ we have $ B\leq C_{i}. $\\
(ii) Every $ S^{c}_{0} $-formula is primary. For $ n\geq 1, $ the formula $ \phi_{(A,B)}(\bar{x})\in S^{c}_{n} $ (which is of the form of the elements of (\ref{eqnScn})) is primary if for each $ 1\leq i\leq m $ the formula $ \theta_{(B,C_{i})}(\bar{x}\bar{y}) $ is a primary $ P^{c}_{n-1} $-formula.\\
\end{dfn}

Towards simplifying formulas dealt with in the proof of \Cref{thmQE}, next lemma enables us to further reduce $ S^{c} $-formulas to the primary ones. The proof of this lemma suggests a useful interaction between the notions of $ \cl $ and $ \cl^{*}. $ In fact the idea is to eliminate the negative part of $ \leqst{}{} $-intrinsic extensions, namely the $ \cl $-part, by replacing them with existential closure formulas. The adequacy of this approach is guaranteed by \Cref{factUpperBound} which introduces a uniform bound on the number of copies of $ \leq $-intrinsic extensions.
\begin{lma}\label{lmaPrimaryFormulas}
Every closure formula is equivalent to a disjunction of primary $ S^{c} $-formulas.
\end{lma}
\begin{proof}
According to \Cref{crlCanonicalForm} and \Cref{dfnPrimaryClosureFormula} we only need to show that any $ P^{c}_{n} $-formula $ \phi_{(A,B)}(\bar{x}), $ is equivalent to a disjunction of primary $ S^{c} $-formulas. A formula $ \phi_{(A,B)}(\bar{x})\in P^{c}_{n} $ is of the form of the elements of (\ref{eqnPcn}). We first show that we may assume that $ A\leq B. $ So suppose on the contrary that $ A\not{\leq}B. $ Then define $ B_{1}:=\cl_{B}(A) $ and let $ B'=B\backslash B_{1}. $ Therefore $ B_{1}\neq B $ and $ \intCl[\leq]{A}{B_{1}}. $ Also it is obvious that $ \delta(B'/B_{1})=0. $ Now by \Cref{factUpperBound} there is a function $ \mu $ which uniformly bounds the number of copies of $ B_{1} $ over $ A. $ Set $ \eta = \mu(\card[A],\card[B']). $ The formula $ \phi_{(A,B)}(\bar{x}) $ is equivalent to the following formula
\[ \existsScaled{1.5}^{\leq\eta}\bar{y}\Big[\diagCl{B_{1}}{\bar{x}}{\bar{y}}\wedge\diagClIAllTupLParenth[B_{1}]{B'}{}{\bar{x}\bar{y}}{}{z}{}{B}\bigvee_{i=1}^{m}\theta_{i}(\bar{x}\bar{y}\bar{z})\Big)\Big],\]
with $ \card[\bar{y}]=\card[B_{1}], \card[\bar{z}]=\card[B'] $ and $ B_{1}\leq B=B_{1}B'. $ Moreover $ \exists^{\leq\eta} $ can be obtained by a finite disjunction of existential closure formulas. Now by an easy induction on $ n $ without loss of generality we can assume that $ A\leq B. $

Now consider that in $ \phi_{(A,B)}(\bar{x}) $ (which is in the form of the elements of (\ref{eqnPcn})) each $ \theta_{(B,C_{i})}(\bar{x}\bar{y}) $ is of the following form
\[ \diagClIExTupLParenth[B]{C}{}{\bar{x}\bar{y}}{}{z}{i}{B}\bigwedge_{j=1}^{n_{i}}\theta'_{ij}(\bar{x}\bar{y}\bar{z}_{i})\Big), \]
where $ \theta'_{ij}(\bar{x}\bar{y}\bar{z}_{i})\in P^{c}_{n-2}. $ Enumerate $ C_{i} $s by $ C_{1},\ldots,C_{k},\ldots C_{m}, $ in such a way that for each $ 1\leq i\leq k $ we have $ B\not{\leq}C_{i} $ and for each $ k< i\leq m $ we have $ B\leq C_{i}. $ Now define for each $ 1\leq i\leq k $ the following structures
\begin{align*}
& E_{i}:=cl_{C_{i}}(A), & C'_{i}:=C_{i}\backslash E_{i},\\
& B_{i}=B\cap E_{i},  & B'_{i}:=B\backslash E_{i},
\end{align*}
and note that $ E_{i}\leq E_{i}C'_{i} $ and $ B'_{i}\subseteq E_{i}C'_{i}, $ hence $ \delta(B'_{i}/B_{i})\geq 0. $ Since $ \intCl{A}{B} $ we have $ \delta(B'_{i}/B_{i})=0. $

Moreover notice that $ \freeJoin{E_{i}}{B_{i}}{B'_{i};} $ since otherwise if $ r\geq 1 $ is the number of the relations that avoid this free amalgamation, then we have
\[ \delta(B'_{i}/E_{i})=\delta(B'_{i}/B_{i})-r\leq -1<0, \]
which contradicts the fact that $ E_{i}\leq E_{i}C'_{i}. $ The fact that $ E_{i} $ is free from $ B'_{i} $ over $ B_{i}, $ implies that for any $ 1\leq i\leq k $ and any two realizations of $ B $ over $ A, $ say $ B^{1} $ and $ B^{2}, $ with $ B^{1}\cap B^{2}= B_{i}, $ a realization of $ E_{i} $ over $ B^{1} $ guarantees the existence of a realization of $ E_{i} $ over $ B^{2}. $

Also since $ \intCl[\leq]{A}{E_{i}} $ again by \Cref{factUpperBound} there is a function $ \mu $ which uniformly bounds the number of copies of $ E_{i} $ over $ A. $ Let $ \eta_{i} = \mu(\card[A],\card[E_{i}]). $ Now it is easy to check that $ \phi_{(A,B)}(\bar{x}) $ is equivalent to the following formula
\[ \existsScaled{1.5}^{\leq\eta_{1}}\bar{z}_{1}\ldots\existsScaled{1.5}^{\leq\eta_{k}}\bar{z}_{k}\Big[\bigwedge_{i=1}^{k}\diagCl[A]{E_{i}}{\bar{x}}{\bar{z}_{i}}\wedge\gamma(\bar{x}\bar{z}_{1}\cdots\bar{z}_{k})\Big], \]
where $ \gamma(\bar{x}\bar{z}_{1}\cdots\bar{z}_{k}) $ is as follows
\[ \diagClIAllTupLParenth{B}{x}{}{}{y}{}{B}\Big[\hspace*{-5pt}\bigvee_{i=k+1}^{m}\hspace*{-7pt}\theta_{i}(\bar{x}\bar{y})\vee\bigvee_{i=1}^{k}\big(``\bar{y}\cap\bar{z}_{i} =B_{i}"\wedge\diagClIExTupLParenth[E_{i}]{C'}{}{\bar{x}\bar{z}_{i}}{}{w}{i}{}\big)\Big]\Big). \]
Note that $ ``\bar{y}\cap\bar{z}_{i} =B_{i}" $ can be expressed by a finite disjunction of complete diagrams.
\end{proof}

The crucial part of the induction that is used in the proof of \Cref{thmQE} is to demonstrate for a closure formula $ \psi(\bar{x}\bar{y})\in\hrchy{} $ that formula $ \exists\bar{y}\psi(\bar{x},\bar{y}) $ also lies in $ \hrchy{}. $ The idea of the proof is to extract from the diagrams appearing in $ \exists\bar{y}\psi(\bar{x},\bar{y}), $ the maximal possible part that can be expressed by a closure formula, say $ \Phi_{\psi}(\bar{x}), $ and to show that the satisfaction of $ \Phi_{\psi}(\bar{x}) $ implies the satisfaction of $ \exists\bar{y}\psi(\bar{x},\bar{y}). $ In fact for a given $ \bar{a}\in M^{*} $ satisfying $ \Phi_{\psi}(\bar{x}), $ the genericity of $ M^{*} $ enables us to find a suitable $ \bar{b}\in M^{*} $ that satisfies $ \psi(\bar{a},\bar{b}). $

The following technical lemma is required for defining $ \Phi_{\psi}(\bar{x}). $
\begin{lma}\label{lmaNegativeCore}
Let $ E\subseteq C,D\in\classPosit{} $ with $ \intCl{E}{D}, E\leq D $ and $ \leqst{E}{C}. $ Then there exists a finite number of structures $ CG_{1},\ldots,CG_{\mu} $ and a finite number of $ \hrchy[s]{0} $-formulas $ \phi_{1}(\bar{x}\bar{v}_{1}\bar{w}),\ldots,\phi_{\mu}(\bar{x}\bar{v}_{\mu}\bar{w}) $ with $ \card[\bar{x}]=\card[E], \card[\bar{w}]=\card[D] $ and $ \card[\bar{v}_{s}]=\card[G_{s}\backslash E] $ for each $ 1\leq s\leq\mu, $ such that the following hold:

$ \bullet $ For any $ s, $ we have that $ \intCl{E}{G_{s}} $ and the formula $ \phi_{s}(\bar{x}\bar{v}_{s}\bar{w}) $ implies $ \diagCl[E]{D}{\bar{x}}{\bar{w}} $ in $ M^{*}. $

$ \bullet $ For any embedding $ \iota:C\longrightarrow M^{*} $ there exist some $ 1\leq s\leq\mu $ such that $ M^{*}\models\diagClIExTupLParenth[\iota E]{G}{}{\iota E	}{}{v}{s}{} $ (witnessed by some $ \bar{g}_{s} $) and for every embedding $ \eta:D\longrightarrow M^{*} $ with $ \eta\restriction_{E}=\iota\restriction_{E} $ we have
\[ M^{*}\models\phi_{s}(\eta E,\bar{g}_{s},\eta D) \quad\Longleftrightarrow\quad \iota C\sqcup_{\iota E}\eta D. \]
\end{lma}
\begin{proof}
First note that there are finite number of possibilities $ CD_{1},\ldots,CD_{m} $ such that $ D_{i}\cong_{E}D $ and having $ D_{i} $ not free over $ E $ with respect to $ C $  for each $ 1\leq i\leq m. $ Hence for each $ i $ we have $ \intCl{E}{D_{i}}, E\leq D_{i} $ and $ \delta(D_{i}/C)<\delta(D_{i}/E)=0. $ Let $ D'_{i}=\cl_{CD_{i}}(D_{i})\backslash C\subseteq D_{i} $ and notice that $ D'_{i}\neq\emptyset $ and $ C\leq_{i}CD'_{i}. $ Thus by \Cref{factUpperBound} we have that $\chi_{M^{*}}(CD'_{i}/C)<\mu_{i} $ for some natural number $ \mu_{i}. $ Therefore one can find a natural number $ \hat{\mu}\leq\prod_{i}\mu_{i} $ and pairs $ \langle G_{1},\mathcal{H}_{1}\rangle,\ldots,\langle G_{\hat{\mu}},\mathcal{H}_{\hat{\mu}}\rangle $ satisfying the following conditions:

$ \bullet $ For any $ 1\leq s\leq\hat{\mu} $ we have $ G_{s}\in\classPosit{}, \intCl{E}{G_{s}} $ and $ G_{s} $ is obtained by a union of copies of $ D $ over $ E $ with $ \card[G_{s}]\hspace*{2pt}\leq\Sigma_{i=1}^{m}\mu_{i}\card[D_{i}]. $ Also $ \mathcal{H}_{s} $ is a family of nonempty substructures of $ G_{s}. $

$ \bullet $ For any embedding $ \iota:C\longrightarrow M^{*} $ there exist some $ s $ and an embedding $ \bar{\iota}_{s}:G_{s}\longrightarrow M^{*} $ with $ \bar{\iota}_{s}\restriction_{E} = \iota\restriction_{E} $ such that for any copy $ \hat{D} $ of $ D $ over $ \iota E $ we have that $ \hat{D}\not\sqcup_{\iota E}\iota C $ if and only if there exists a copy $ D^{\dag}\subseteq \bar{\iota}_{s}(G_{s}) $ of $ D $ over $ \bar{\iota}_{s} E $ such that
$  \cl_{\iota C D^{\dag}}(\iota C)=\cl_{\iota C\hat{D}}(\iota C)\in\mathcal{H}_{s}.  $

Now one can see that for any embedding $ \iota:C\longrightarrow M^{*} $ one of the structures $ G_{s} $ is realized over $ \iota C, $ and for every embedding $ \eta:D\longrightarrow M^{*} $ its image $\eta D $ is realized free from $ \iota C $ over $ \iota E $ if and only if $ \eta D\cap G_{s}\notin\mathcal{H}_{s}. $ Clearly the latter property is expressible by a $ \hrchy[s]{0} $-formula which introduces the desired $ \phi_{s}(\bar{x}\bar{v}_{s}\bar{x}). $

To finish the proof, note that for any $ s $ there are finitely many complete diagrams over $ C\cup G_{s} $ extending both $ C $ and $ G_{s}. $ Therefore by assigning  the same $ \phi_{s}(\bar{x}\bar{v}_{s}\bar{w}) $ to all the possible diagrams of $ CG_{s} $ we may get a new list $ \langle CG_{1},\phi_{1}(\bar{x}\bar{v}_{1}\bar{w})\rangle,\ldots,\langle CG_{\mu},\phi_{\mu}(\bar{x}\bar{v}_{\mu}\bar{w})\rangle $ for some $ \mu\geq\hat{\mu} $ as required.

\end{proof}

\paragraph{Definition of $ \Phi_{\psi}(\bar{x}). $} For any $\psi(\bar{x},\bar{y})\in S^{c}_{n} $ we define a closure formula $ \Phi_{\psi}(\bar{x}) $ (that will be proved to be equivalent to $ \exists\bar{y}\psi(\bar{x},\bar{y}) $ in \Cref{thmQE}). The definition is given by induction on $ n. $ Suppose that $ \psi(\bar{x}\bar{y})\in S^{c}_{0}, $ hence for some $ A\subseteq B\in\classPosit{} $ it is of the form of $ \diag_{AB}(\bar{x}\bar{y}). $ Set $ E:=\cl^{*}_{B}(A) $ and let
$ \Phi_{\psi}(\bar{x}) $ be $
\diagClIExTupLParenth{E}{x}{}{}{z}{}{}. $

For $ n\geq 1 $ suppose that the innermost $ S^{c}_{3} $-formulas appearing in $ \psi(\bar{x}\bar{y}) $ are of the following form
\begin{align}\label{eqnSigma3}
\begin{split}
&\diagClIExTupLParenth[A^{*}B^{*}]{C}{}{\bar{x}^{*}\bar{y}^{*}}{}{z}{}{gg}\bigwedge_{j=1}^{m}
\diagClIAllTupLParenth[C]{D}{}{\bar{x}^{*}\bar{y}^{*}\bar{z}}{}{w}{j}{neg}\wedge\\
&\qquad\diagClIAllTupLParenth[C]{D}{}{\bar{x}^{*}\bar{y}^{*}\bar{z}}{}{w}{0}{B}\bigvee_{k=1}^{n}\diagClIExTupLParenth[D_{0}]{C}{}{\bar{x}^{*}\bar{y}^{*}\bar{z}\bar{w}_{0}}{}{z}{0k}{}\Big)\bigg),
\end{split}
\end{align}
in which $ \bar{x}\bar{y}\subset \bar{x}^{*}\bar{y}^{*}. $ Moreover by induction we may assume that $ A^{*}\leq^{*}\intCl{B^{*}}{C} $ and $ \intCl{A}{A^{*}}. $

Now by cutting all the leaves and their immediate predecessors from $ \tau_{\psi} $ we obtain a tree $ \tau' $ which corresponds to a formula $ \psi'(\bar{x}\bar{y}). $ This formula belongs to a lower level of the hierarchy and in fact we have $ \tau'=\tau_{\psi'}. $ Hence by induction hypothesis one can define a closure formula $ \Phi_{\psi'}(\bar{x}). $ We may further assume by induction that the innermost $ S^{c}_{1} $-formulas appearing in $ \Phi_{\psi'}(\bar{x}) $ are formulas of the following form
\[ \diagClIExTupLParenth[A^{*}]{E}{}{\bar{x}^{*}}{\bar{z}_{0}}{}{}{},\]
where $ E=\cl^{*}_{B^{*}C}(A^{*})\subseteq C. $ Let us denote these formulas by $ \gamma(\bar{x}^{*}). $

Now in order to define $ \Phi_{\psi}(\bar{x}) $ we replace each formula $ \gamma(\bar{x}^{*}) $ in $ \Phi_{\psi'}(\bar{x}) $ with a new formula $ \theta(\bar{x}^{*}) $ defined as follows. Let $ C':= B^{*}C\backslash E $ and for $ 1\leq j\leq m $ select those $ D_{j} $s which are \oldClLiteral{}{E}{C'} and renumerate them by $ D_{1},\ldots,D_{m'}. $ We have two different cases:

\textbf{Case 1,} \notOldClLiteral{D_{0}}{E}{C':}{neg} In this case let $ \theta(\bar{x}^{*}) $ be:
\[ \diagClIExTupLParenth[A^{*}]{E}{}{\bar{x}^{*}}{\bar{z}_{0}}{}{}{B}\bigwedge_{j=1}^{m'}\diagClIAllTupLParenth[E]{D}{}{\bar{x}^{*}\bar{z}_{0}}{}{w}{j}{neg}\Big). \]

\textbf{Case 2,} \oldClLiteral{D_{0}}{E}{C':} In this case for each $ 1\leq k\leq n $ set $ E_{0k}:= \closure{C_{0k}}{D_{0}}. $ Now obtain $ C'G_{1},\ldots,C'G_{\mu} $ by applying \Cref{lmaNegativeCore} to $ E, D_{0} $ and $ EC', $ and for any $ 1\leq s\leq\mu $ set $ F_{s}:=\cl_{C'G_{s}}^{*}(G_{s}). $ Moreover let $ \sigma_{s}(\bar{x}^{*}\bar{z}_{0}) $ be the following formula
\begin{align*}
\qquad\qquad\diagClIExTupLParenth[E]{F}{}{\bar{x}^{*}\bar{z}_{0}}{}{u}{s}{B}&\\
&\hspace*{-70pt}\forall\bar{w}_{0}\Big[\phi_{s}(\bar{x}^{*}\bar{v}_{s}\bar{w}_{0})\rightarrow\bigvee_{k=1}^{n}\diagClIExTupLParenth[D_{0}]{E}{}{\bar{x}^{*}\bar{z}_{0}\bar{w}_{0}}{}{z}{0k}{}\Big]\Big).
\end{align*}
Since $ \phi_{s}(\bar{x}^{*}\bar{v}_{s}\bar{w}_{0}) $ implies $ \diagCl[E]{D_{0}}{\bar{x}^{*}}{\bar{w}}, $ by part (iii) of \Cref{lmaHierarchyCanonicalFormSigma} the above formula is a $ \hrchy[s]{3} $-formula.
Now let $ \theta(\bar{x}^{*}) $ be the following formula
\begin{align*}
\diagClIExTupLParenth[A^{*}]{E}{}{\bar{x}^{*}}{\bar{z}_{0}}{}{}{B}\bigwedge_{j=1}^{m'}\diagClIAllTupLParenth[E]{D}{}{\bar{x}^{*}\bar{z}_{0}}{}{w}{j}{neg}\wedge
\bigvee_{s=1}^{\mu}\sigma_{s}(\bar{x}^{*}\bar{z}_{0})\Big).
\end{align*}
\begin{thm}\label{thmQE}
$ T^{*}=\theory(M^{*}) $ admits quantifier elimination down to closure formulas. More precisely, each $\mathcal{L}$-formula $ \frmI{x}{} $ is equivalent to a closure formula in $ M^{*} $.
\end{thm}
\begin{proof}

We proceed by induction on the complexity of $ \frmI{x}{} $. The induction base for atomics as well as the induction steps for negation, conjunction and disjunction are obvious. Therefore based on \Cref{lmaPrimaryFormulas} we may suppose that $ \frmI{x}{} $ is of the form of $ \frmIExTup{x}{}{y}{} $ where $ \frmDblI{x}{}{y}{}$ is a primary $ S^{c}_{n} $-formula of the form of \Cref{eqnSigmaWLG}.

Now by induction on $ n $ we show that
\begin{equation}\label{eqnMainPartOfTheorem} M^{*}\models\forall\bar{x}(\Phi_{\psi}(\bar{x})\leftrightarrow\exists\bar{y}\frmDblI{x}{}{y}{}).
\end{equation}

To see that $ \exists\bar{y}\frmDblI{x}{}{y}{} $ implies $ \Phi_{\psi}(\bar{x}), $ using the notations employed in defining $ \Phi_{\psi}(\bar{x}), $ note that for any $ \bar{a},\bar{b}\in M^{*} $ we have that $ M\models\psi(\bar{a}\bar{b})\rightarrow\psi'(\bar{a}\bar{b}). $ Also by induction hypothesis we know that $ \psi'(\bar{a}\bar{b}) $ implies $ \Phi_{\psi'}(\bar{a}). $ But since each innermost $ \hrchy[s]{3} $-formula appearing in $ \psi(\bar{x}\bar{y}) $ logically implies the corresponding introduced $ \theta(\bar{x}^{*}) $, we have that $ M^{*}\models\Phi_{\psi}(\bar{a}). $

For the other direction, assume for some $ \bar{a}\in M^{*} $ that $ M^{*}\models\Phi_{\psi}(\bar{a}). $
To ease the notation we omit subscript $ M^{*} $ from $ \cl^{*}_{M^{*}}(\bar{a}) $ and write $ \cl^{*}(\bar{a}) $ instead.
\subparagraph{Claim.}\hspace{-3pt}There is a structure $ \mathbb{C}_{\bar{a},\psi}\in\classPosit{} $ satisfying the following properties:

$\bullet\hspace*{4pt} \cl^{*}(\bar{a})\leq^{*}\mathbb{C}_{\bar{a},\psi}. $

$ \bullet $ For any strong embedding of $ \mathbb{C}_{\bar{a},\psi} $ into $ M^{*} $ over $ \cl^{*}(\bar{a}), $ there exist some $ \bar{b}\in\mathbb{C}_{\bar{a},\psi} $ with $ M^{*}\models\psi(\bar{a}\bar{b}). $

\begin{proofAV} We define the structure $ \mathbb{C}_{\bar{a},\psi} $ by induction on $ n $. For $ n=0 $ let $ \mathbb{C}_{\bar{a},\psi} $ be the structure $ B\backslash \cl^{*}(\bar{a}), $ and for $ n=1 $ define it to be $ C\backslash \cl^{*}(\bar{a}). $

For $ n\geq 1 $ following the notations used in the definition of $ \Phi_{\psi}(\bar{x}), $ we obtain the formula $ \psi'(\bar{x}\bar{y}) $ and a tree $ \tau'=\tau_{\psi'}. $ Now let $ C_{1},\ldots,C_{\lambda} $ be an enumeration of the diagrams appearing as the leaves of $ \tau_{\psi'}. $ By induction hypothesis we know that $ \mathbb{C}_{\bar{a},\psi'} $ is already defined.

Note that  corresponding to each $ C_{i} $ there is a $ \hrchy[s]{3} $-formula $ \psi_{(A_{i}^{*}B_{i}^{*},C_{i})}(\bar{x}^{*}\bar{y}^{*}) $ in the form of \Cref{eqnSigma3}. To avoid any ambiguity we attach an $ ``i" $ at the beginning of all the subscripts used in the definition of $ \Phi_{\psi}(\bar{x}). $

Now for each $ 1\leq i\leq\lambda $ enumerate by $ D^{1}_{i0},\ldots,D^{\nu_{i}}_{i0} $ the actual realizations of $ D_{i0} $ over $ E_{i} $ in $ M^{*}. $ Moreover suppose for each $ 1\leq t\leq\nu_{i} $ that $ E^{t}_{i0k} $ is one of the structures $ E_{i0k} $ which is forced by $ \Phi_{\psi}(\bar{a}) $ to be realized over $ D^{t}_{i0}. $ Let $ F_{i} $ be the structure $ E_{i}\cup\bigcup_{t=1}^{\nu_{i}}E^{t}_{i0k} $ forced by $ \Phi_{\psi}(\bar{a}) $ to be realized over $ \bar{a} $ as a subset of $ \cl^{*}(\bar{a}), $ and consider $ H_{i} $ to be a structure isomorphic to $ C'_{i}. $ Then for each $ 1\leq t\leq\nu_{i} $ let $ H^{t}_{i0k} $ be a structure which is isomorphic to $ C'_{i0k}, $ and define $ G_{i} $ to be the following structure
\[ H_{i}\cup\bigcup_{t=1}^{\nu_{i}}H^{t}_{i0k}, \]
in which $ H^{t}_{i0k} $s are mutually freely joined over $ H_{i} $ and for each $ t $ we have $ H^{t}_{i0k}E^{t}_{i0k}\cong C'_{i0k}\cup E_{i0k} $ and
\[ R(G_{i},\cl^{*}(\bar{a}))=\bigcup_{t=1}^{\nu_{i}}R(H_{i}H^{t}_{i0k},E^{t}_{i0k}). \]

Denote by $ G $ the structure whose universe is
$ G:=\bigcup_{i=1}^{\lambda}G_{i} $ in which all $ G_{i} $s are mutually freely joined over $ \cl^{*}(\bar{a}) $ having that
\[ R(G,\cl^{*}(\bar{a}))=\bigcup_{i=1}^{\lambda}R(G_{i},F_{i}). \]

Now we show that $ \leqst{\cl^{*}(\bar{a})}{G}. $ For each $ i $ and $ t $ suppose that $ N_{i} $ and $ K^{t}_{i} $ be arbitrary subsets of $ H_{i} $ and $ H^{t}_{i0k} $ respectively, and note that according to the way $ G $ is defined we have the following equations
\begin{align*}
\delta\Big(\bigcup_{i=1}^{\lambda}(N_{i}\cup\bigcup_{t=1}^{\nu_{i}}K^{t}_{i})/\cl^{*}(\bar{a})\Big) &= \sum_{i=i}^{\lambda}\delta(N_{i}\cup\bigcup_{t=1}^{\nu_{i}}K^{t}_{i})/F)\\
&= \sum_{i=i}^{\lambda}\delta(N_{i}\cup\bigcup_{t=1}^{\nu_{i}}K^{t}_{i})/F_{i}),
\end{align*}
but for each $ i $ we have
\[ \delta(N_{i}\cup\bigcup_{t=1}^{\nu_{i}}K^{t}_{i})/F_{i})=\delta(\bigcup_{t=1}^{\nu_{i}}K^{t}_{i}/N_{i}F_{i})+\delta(N_{i}/F_{i}). \]
Moreover
\[ \delta(N_{i}/F_{i})=\sum_{t=1}^{\nu_{i}}\delta(N_{i}/E^{t}_{i0k}), \]
and because $ \freeJoin{C'_{i}}{E_{i}}{D_{i0}} $ and $ \freeJoin{C'_{i}}{D_{i0}}{E_{i0k}}, $ we have $ \delta(N_{i}/E^{t}_{i0k})=\delta(N_{i})>0. $

Now notice that
\[ \delta(\bigcup_{l=1}^{\nu_{i}}K^{t}_{i}/N_{i}F_{i})=\sum_{l=1}^{\nu_{i}}\delta(K^{t}_{i}/N_{i}F_{i})= \sum_{l=1}^{\nu_{i}}\delta(K^{t}_{i}/N_{i}F_{i_{l}}),\]
and since $ \psi(\bar{x}\bar{y}) $ is considered to be primary, for each $ i $ and $ k $ we have $ D_{i0}\leq C_{i0k}, $ hence for each $ t $ we have $ \delta(K^{t}_{i}/N_{i}E^{t}_{i0k})\geq 0. $ Therefore we have
\begin{equation}\label{eqnLastStep}
 \delta\Big(\bigcup_{i=1}^{\lambda}(N_{i}\cup\bigcup_{l=1}^{\nu_{i}}K^{t}_{i})/F\Big)>0,
\end{equation}
and hence $ \cl^{*}(\bar{a})\leq^{*} G. $ Now set $ \mathbb{C}_{\bar{a},\psi}:= \mathbb{C}_{\bar{a},\psi'}\cup G. $
\end{proofAV}
Finally since $ M^{*} $ is generic, $ \mathbb{C}_{\bar{a},\psi} $ can be strongly embedded over $ \cl^{*}(\bar{a}). $ The existence of a tuple $ \bar{b}\subseteq\mathbb{C}_{\bar{a},\psi} $ with the property that $ M^{*}\models\psi(\bar{a}\bar{b}) $ is obvious by noticing that in the definition of $ \Phi_{\psi}(\bar{x}), $ for each $ i $ we only take into account those $ D_{i0} $s that are \oldClLiteral{}{E_{i}}{C'_{i};} and since we embed $ \mathbb{C}_{\bar{a},\psi} $ over $ \cl^{*}(\bar{a}) $ in an old way, the $ D_{i0} $s that are not \oldClLiteral{}{E_{i}}{C'_{i}} can not be satisfied at all. The situation is similar for those $ 1\leq j\leq m_{i} $ for which we have $ \notOldCl{D_{ij}}{E_{i}}{C'_{i}}. $

\end{proof}

An important difference between $ M^{*} $ and an arbitrary semigeneric structure $M^{\prime}$ is that while the closure of any finite tuple is finite in the former, it can be infinite in the latter. As one can see in the proof of \Cref{thmQE}, we strongly made use of this property inside $ M^{*}. $ Hence a further clarification of the distinctions between $ \theory(M^{*}) $ and $ T^{*}_{\sem}, $ can be obtained by answering the following question.
\begin{qst}
Does $ T^{*}_{\sem} $ admit a quantifier elimination down to closure formulas?
\end{qst}

The next corollary shows that the type of a finite tuple is completely determined by its closure. This corollary also generalizes Corollary 2.4 of \cite{Wagner-Relational} and Lemma 1.30 of \cite{Baldwin&Shelah-Randomness} to the context of $ \classPositOpr{}. $
\begin{crl}\label{crlType}
Suppose that $ M $ is a saturated model of $ T^{*}. $ Let $ \bar{a},\bar{a}'\in M $ both have a same diagram $ A. $ Then we have:
\[ \closure{M}{\bar{a}}\cong_{A}\closure{M}{\bar{a}'}\quad\iff\quad \type^{M}(\bar{a})=\type^{M}(\bar{a}'). \]
\end{crl}
\begin{proof}
We prove that having isomorphic closures implies equality of types, and the other direction is rather obvious. So note that if $ \closure{M}{\bar{a}}\cong_{A}\closure{M}{\bar{a}'}, $ then for any $ n\in\omega $ and any $ \phi(\bar{x})\in\hrchy[h]{n}, $ we have that $ M\models\phi(\bar{a})\leftrightarrow\phi(\bar{a}'). $ Hence by \Cref{thmQE} we have $ \type^{M}(\bar{a})=\type^{M}(\bar{a}'). $
\end{proof}

In the rest of this section we show that \Cref{thmQE} is true for a certain ultraproduct of models each of which is elementarily equivalent to a generic structure. Using this result we prove that the theory of this ultraproduct is the same as $ T^{*}.$ In fact it is proved in \cite{Brody&Laskowski-OnRationalLimits} that if $ \{\alpha_{n}\}_{n\in\omega}\subseteq (0,1] $ is an strictly decreasing sequence of real numbers converging to a rational $ \alpha\in(0,1], $ and for each $ n\in\omega $ we consider a structure $ N_{\alpha_{n}} $ elementarily equivalent either to $ M_{\alpha_{n}} $ or to $ M^{*}_{\alpha_{n}}, $ and $ \mathcal{U} $ is a non-principal ultrafilter on $ \omega, $ then the ultraproduct $ \prod_{\mathcal{U}}N_{\alpha_{n}} $ is a model of $ T^{*}_{\alpha\mhyphen \sem}. $ We show that \Cref{thmQE} holds in $  \prod_{\mathcal{U}}N_{\alpha_{n}} $ using which we prove that $ \prod_{\mathcal{U}}N_{\alpha_{n}} $ is elementarily equivalent to $ M^{*}_{\alpha}. $

In \Cref{factUpperBound}, for each $ \alpha $ we introduced an upper bound on $ \mu_{\alpha}, $ say $ \eta_{\alpha}. $ Accordingly, it is easily seen that for every $ \alpha<\beta\in(0,1] $ we have $ \eta_{\alpha}>\eta_{\beta}. $ Hence in particular for a strictly decreasing sequence $ \{\alpha_{n}|n\in\omega\} $ converging to some $ \alpha\in(0,1]\cap\mathbb{Q}, $ if $ A\leq^{*}_{i,\alpha}B\in\classPosit{\alpha} $ and $ \delta_{\alpha}(B/A)=0, $ then there exists a natural number $ n_{0} $ such that for every $ n\geq n_{0} $ we have $ A\leq_{i,\alpha_{n}}B\in\classPositEq{\alpha_{n}} $ and $ \delta_{\alpha_{n}}(B/A)<0. $ Moreover for any $ n\geq n_{0} $ we can observe that $ \eta_{\alpha_{n+1}}>\eta_{\alpha_{n}}. $ This phenomena seems natural considering the fact that $ \eta_{\alpha} $ is actually infinity.

Note that the notions of intrinsic extension and strong embedding change from $ \classPosit{\alpha} $ to $ \classPosit{\alpha_{n}} $s. Therefore we denote by  $ \hrchySeq[h]{}{\alpha_{n}}, $  the set of closure formulas defined for $ \classPosit{\alpha_{n}} $ and $ M^{*}_{\alpha_{n}}. $ However we have the following lemma.

\begin{lmaB}\label{lmaSeqProp}
(i) For any structure $ A\in\classPosit{\alpha} $ there exists a natural number $ n_{A} $ such that for any $ n\geq n_{A} $ we have $ A\in\classPosit{\alpha_{n}}\subseteq\classPositEq{\alpha_{n}}. $\\
(ii) For any structures $ A\leq^{*}_{\alpha} B\in\classPosit{\alpha} $ there exists a natural number $ n_{A,B} $ such that for any $ n\geq n_{A,B} $ we have $ A\leq^{*}_{\alpha_{n}} B, $ and hence $ A\leq_{\alpha_{n}} B. $\\
(iii) For any structures $ A\leq^{*}_{i,\alpha} B\in\classPosit{\alpha} $ there exists a natural number $ n_{(A,B)} $ such that for any $ n\geq n_{(A,B)} $ we have $ A\leq_{i,\alpha_{n}} B, $ and hence $ A\leq^{*}_{i,\alpha_{n}} B. $\qed
\end{lmaB}
\begin{crl}\label{crlSeqClosure}
For any structures $ A\subseteq B\in\classPosit{\alpha}, $ there is a natural number $ m $ such that for any $ n\geq m $ we have $ \alpha\mhyphen\cl^{*}_{B}(A)= \alpha_{n}\mhyphen\cl^{*}_{B}(A). $\qed
\end{crl}
\begin{thmB}\label{thmProductQE}
(i) Suppose that there exists a natural number $ m $ such that for any $ n\geq m $ we have $ \alpha_{n}\in\mathbb{Q} $ and $ N_{\alpha_{n}} $ is elementary equivalent to the $ \classPositOpr{\alpha_{n}} $-generic. Then for each formula $ \phi(\bar{x})\in\mathcal{L} $ there exists a closure formula $ \theta_{\phi}(\bar{x})\in\hrchySeq[h]{}{\alpha} $ such that
\[ \prod_{\mathcal{U}}N_{\alpha_{n}}\models\forall\bar{x}(\phi(\bar{x})\leftrightarrow\theta_{\phi}(\bar{x})). \]
\\(ii) $ \prod_{\mathcal{U}}N_{\alpha_{n}}\equiv M^{*}_{\alpha}. $
\end{thmB}
\begin{proof}

(i) By \Cref{thmQE} for each $\phi(\bar{x})\in\mathcal{L} $ there is a closure formula $ \theta_{\phi}(\bar{x})\in\hrchySeq[h]{}{\alpha} $ such that $ M^{*}_{\alpha}\models\forall\bar{x}(\phi(\bar{x})\leftrightarrow\theta_{\phi}(\bar{x})). $ Recall that the main step in proving \Cref{thmQE} amounts to defining the closure formula $ \Phi_{\psi}(\bar{x}) $ and proving the expression (\ref{eqnMainPartOfTheorem}). Now consider a given formula $ \phi(\bar{x}) $ in the form of $ \exists\bar{y}\psi(\bar{x},\bar{y}) $ with $ \psi(\bar{x},\bar{y})\in\hrchySeq[h]{}{\alpha}. $ First note that based on \Cref{lmaSeqProp} there is a natural number $ m_{\psi} $ such that for any $ n\geq m_{\psi} $ we have $ \psi(\bar{x},\bar{y})\in\hrchySeq[h]{}{\alpha_{n}}.$ Moreover based on \Cref{crlSeqClosure} there exists a natural number $ m'_{\psi} $ such that the closure of structures in $ \classPosit{\alpha_{n}} $ remains unchanged for all $ n\geq m'_{\psi}. $ Hence for all $ \alpha_{n}\leq \alpha_{m'_{\psi}} $ the procedure of defining $ \Phi_{\psi}(\bar{x}) $ leads to the same closure formula as the closure formula we obtain for $ \alpha. $ Furthermore \Cref{thmQE} holds for an arbitrary fixed $ \alpha\in(0,1]\cap\mathbb{Q}. $ So there exists a natural number $m_{\phi}=\max\{m,m_{\psi},m'_{\psi}\}$ such that for any $ n\geq m_{\phi}$ we have $ N_{\alpha_{n}}\models\forall\bar{x}(\phi(\bar{x})\leftrightarrow\theta(\bar{x})). $ Therefore
\[\prod_{\mathcal{U}}N_{\alpha_{n}}\models\forall\bar{x}(\phi(\bar{x})\leftrightarrow\theta(\bar{x})). \]
(ii) Let $ M $ and $ N $ be respectively saturated models of $ T^{*}_{\alpha} $ and $ \theory(\prod_{\mathcal{U}}N_{\alpha_{n}}). $ Based on part (i) and using compactness it is easy to see that the following set defines a back and forth system between the substructures of $ M $ and $ N. $
\begin{align*}
  \mathcal{I}:= \Biggl\{(\bar{a},\bar{a}')\;\Bigg|\; \pctext{70mm}{$\bar{a}\in M,\bar{a}'\in N, \card[\bar{a}]=\card[\bar{a}']\neq 0,\bar{a}\equiv_{0}\bar{a}',$\\$\forall\phi(\bar{x})\in\hrchySeq[h]{}{\alpha}, M\models\phi(\bar{a})\Leftrightarrow N\models\phi(\bar{a}').$}\Biggr\}, \\
\end{align*}
hence we have $ \prod_{\mathcal{U}}N_{\alpha_{n}}\equiv M^{*}_{\alpha}. $
\end{proof}

Note that the hypothesis in the first part of the theorem can be replaced by the following statement: ``for almost all $ n\in\omega $ we have that $ \alpha_{n}\in\mathbb{Q} $ and $ N_{\alpha_{n}} $ is elementary equivalent to the $ \classPositOpr{\alpha_{n}} $-generic''.

\section{Finite Model Property}\label{secFMP}
%

In this section we show that $ T^{*}_{\alpha} $ does not have the finite model property, for each $ \alpha\in(0,1]\cap\mathbb{Q}.$ Towards this end we give a weaker form of Proposition 3.3 of \cite{Brody&Laskowski-OnRationalLimits} which suffices for our purpose. Throughout this section we omit the subscript $ \alpha. $ In fact this lemma is restricted to $ T^{*} $ instead of $ T^{*}_{\sem} $:
\begin{lma}\label{lmaDefinFinRel}
For any $ k\in\omega $ there is a definable relation $ R^{k}(x_{1},\ldots,x_{k},y) $, symmetric in the first $ k $ variables, such that for any $ S\subset_{\omega} M^{*}$ and any $ X\subseteq \left[S\right]^{k} $ there exist some $ v\in  M^{*} $ such that for any $ \bar{a}\in M^{*} $ we have
\[ M^{*}\models R^{k}(\bar{a},v)\quad\quad\iff\quad\quad \{a_{1},\ldots,a_{k}\}\in X \]
\end{lma}

\begin{proof}
See the proof of Proposition 3.3 in \cite{Brody&Laskowski-OnRationalLimits}.
\end{proof}

Based on the above lemma any $ k $-element subset $ X $ of a finite subset $ S\subseteq M^{*} $ has a ``\textit{code}" $ v\in M^{*} $ that defines $ X $ by $ R^{k}(\bar{x},v). $

\begin{rmrk}\label{rmrkBasicOperation}
It is easy to see that one can find new codes for the sets obtained by the basic set-theoretic operations on some given coded sets. To ease later references we fix a notation for the formulas which correspond to some of these operations. So suppose that $ v_{1} $ and $ v_{2} $ code respectively the sets $ X_{1}\subseteq[S_{1}]^{k} $ and $ X_{2}\subseteq[S_{1}]^{k} $ with $ S_{1}\cap S_{2}=\emptyset. $ Now we fix the following:

- Formula $ \upsilon^{k}(v_{1},v_{2},w), $ declaring that $ w $ codes $ X_{1}\cup X_{2}. $

- Formula $ \pi^{k}(v_{1},v_{2},w), $ declaring that $ w $ codes the set $ Y=\{x_{1}\cup x_{2}\hspace*{3pt}|\hspace*{3pt}x_{i}\in X_{i},i=1,2\}. $ Note that since $ S_{1}\cap S_{2}=\emptyset $ the cardinality of $ Y $ is equal to the cardinality of $ X_{1}\times X_{2}. $

- Formula $ \eta^{k}(v_{1},v_{2}), $ which indicates that there is an injection from $ X_{1} $ into $ X_{2}. $

- Formula $ \theta^{k}(v_{1},v_{2}), $ which states that there is a bijection between $ X_{1} $ and $ X_{2}. $
\end{rmrk}

Our aim is to interpret $ \langle\mathbb{Q}^{\geq 0},+,.,<\rangle $ in $ M^{*} $ (\Cref{propQSumMul}). We further generalize the technique used in \cite{Brody&Laskowski-OnRationalLimits} for interpreting Robinson arithmetic in the semigeneric models of $ \classPositOpr{}. $

\paragraph*{Motivation and set up.} \hspace*{-7pt}Fix three finite substructures $ A,B $ and $ C $ of $ M^{*} $ with $ \minPair{A}{B} $, $ \minPair{A}{C} $, $ \delta(B/A)=\delta(C/A)=0 $, and $ B\not{\cong_{A}}C. $ The idea is to interpret a positive rational number $ p/q $ by a copy of $ A, $ over which it is realized $ p $ copies of $ B, $ and $ q $ copies of $ C. $ In fact $ \chi(B/A) $ plays the role of the numerator and $ \chi(C/A) $ that of the denominator.

Now let $ m,n\in\omega $ denote respectively the cardinality of the sets $ B\backslash A $ and $ C\backslash A. $ Let $ \mathbb{A} $ be the set of all $ \bar{a}\in M^{*} $ with the following properties:

$ \bullet\hspace*{3pt}\bar{a}\cong A. $

$ \bullet\hspace*{3pt}\chi(C/A)\geq 1. $

$ \bullet $ The intersection of any two distinct copies of B (resp. copies of $ C $) over A is A.

$ \bullet $ The intersection of a copy of $ B $ over $ A $ with a copy of $ C $ over $ A $ is $ A $.

It is easy to write a formula that defines $ \mathbb{A}, $ let us denote it by $ \phi_{A}(\bar{x}). $

\paragraph*{Notation.} \hspace*{-7pt}We borrow the notion of a basis from \cite{Brody&Laskowski-OnRationalLimits}. So for each $ \bar{a}\in\mathbb{A} $ let $ \mathbb{B}_{\bar{a}} $ and $ \mathbb{C}_{\bar{a}} $ respectively denote the union of all copies of $ B $ and $ C $ over $ \bar{a}. $ A subset $ \mathbf{B}\subseteq\mathbb{B}_{\bar{a}} $ is called a $ B $\textit{-basis} for $ \bar{a} $ if it contains exactly one element of each copy of $ B $ over $ \bar{a} $. Similarly we can define a $ C $\textit{-basis}.

The notion of $ B $-basis can be expressed by a formula $ \beta_{A}(\bar{x},v) $ defined as following
\begin{align*}
&\forall y\Big[R^1(y,v)\rightarrow\Big(\exists y_{2}\cdots y_{m}Diag_{(A,B)}(\bar{x},yy_{2}\ldots y_{m})\wedge\\
&\hspace{8mm}\forall y'\Big((y'\neq y\wedge R^{1}(y',u))\rightarrow\neg\exists y_{3}\cdots y_{m}Diag_{(A,B)}(\bar{x},yy'y_{3}\ldots y_{m})\Big)\Big)\Big].
\end{align*}
In fact $ \beta_{A}(\bar{x},v) $ indicates that $ v $ codes a $ B $-basis for $ \bar{x}. $ Similarly we may set a formula $ \gamma_{A}(\bar{x},v) $ to express the notion of $ C $-basis.

\begin{lma}\label{lmaQOrder}
$ \langle\mathbb{Q},<\rangle $ is interpretable in $ M^{*}. $
\end{lma}
\begin{proof}
The equivalence relation we consider on $ M^{*} $ is defined in such a way that each $ \bar{a}_{1} $ and $ \bar{a}_{2} $ in $ M^{*} $ be equivalent if and only if $ \chi(B/\bar{a}_{1}).\chi(C/\bar{a}_{2}) = \chi(B/\bar{a}_{2}).\chi(C/\bar{a}_{1}). $ So let $ E(\bar{x}_{1},\bar{x}_{2}) $ be the following formula
\begin{align*}
\bigwedge_{i=1}^{2}\phi_{A}(\bar{x}_{i})\wedge\exists \bar{u}\bar{v}\bar{w}&\Big[\bigwedge_{i=1}^{2}\beta_{A}(\bar{x}_{i},u_{i})\wedge\bigwedge_{i=1}^{2}\gamma_{A}(\bar{x}_{i},v_{i})\wedge\\
&\pi^{1}(u_{1},v_{2},w_{1})\wedge\pi^{1}(u_{2},v_{1},w_{2})\wedge\theta^{2}(w_{1},w_{2})\Big],
\end{align*}
where $ \card[\bar{u}]=\card[\bar{v}]=\card[\bar{w}]=2. $ The formula above expresses literally that while $ u_{1} $ and $ u_{2} $ code respectively a basis for $ \mathbb{B}_{a_{1}} $ and $ \mathbb{B}_{a_{2}} $, say $ \mathbf{B}_{1} $ and $ \mathbf{B}_{2} $, $v_{1} $ and $ v_{2} $ define respectively a basis for $ \mathbb{C}_{a_{1}} $ and $ \mathbb{C}_{a_{2}} $, say $ \mathbf{C}_{1} $ and $ \mathbf{C}_{2} $. Furthermore $ w_{1} $ and $ w_{2} $ code respectively $ \mathbf{B}_{1}\times\mathbf{C}_{2} $ and  $ \mathbf{B}_{2}\times\mathbf{C}_{1}, $ and $ w_{3} $ defines a bijection between $ \mathbf{B}_{1}\times\mathbf{C}_{2} $ and $ \mathbf{B}_{2}\times\mathbf{C}_{1} $.

Note that using the genericity of $ M^{*}, $ for any $ p,q,p',q'\in\omega $ with $ q\neq 0 $ and $ q'\neq 0, $ the disjoint union of the representatives of the classes $ p/q $ and $ p'/q' $ is strongly embeddable into $ M^{*}.$ Hence the formulas introduced in \Cref{rmrkBasicOperation} have still their valid meanings. Also since the closures of finite subsets of $ M^{*} $ are finite, it is easy to check that the above formula defines an equivalence relation with the required property.

To define an order on $ \mathbb{A}/E, $ it suffices to replace the bijection used in $ E(\bar{x}_{1},\bar{x}_{2}) $ by an injection. Therefore let $ O(\bar{x}_{1},\bar{x}_{2}) $ be the following formula
\begin{align*}
\bigwedge_{i=1}^{2}\phi_{A}(\bar{x}_{i})\wedge\exists \bar{u}\bar{v}\bar{w}&\Big[\bigwedge_{i=1}^{2}\beta_{A}(\bar{x}_{i},u_{i})\wedge\bigwedge_{i=1}^{2}\gamma_{A}(\bar{x}_{i},v_{i})\wedge\\
&\pi^{1}(u_{1},v_{2},w_{1})\wedge\pi^{1}(u_{2},v_{1},w_{2})\wedge\eta^{2}(w_{1},w_{2})\Big].
\end{align*}
where $ \card[\bar{u}]=\card[\bar{v}]=\card[\bar{w}]=2. $ Now it is easy to check that $ \langle\mathbb{A}/E,O(\bar{x}_{1},\bar{x}_{2})\rangle $ interprets a linear order. Moreover since $ \delta(B/A)=0 $, for any $ p\in\omega $ the predimension of the structure constructed by $ A $ with exactly $ p $ disjoint copies of $ B $, all be mutually freely joined over $ A $, is strictly positive and hence this structure lies in $ \classPosit{} $. The same is true for any $ p,q\in\omega $ with $ q\neq 0 $ and the structure constructed by $ A $ with exactly $ p $ disjoint copies of $ B $ and $ q $ disjoint copies of $ C $ all be mutually freely joined over $ A $. Also note that all of these structures are embedded strongly in $ M^{*}. $ Therefore $ \langle\mathbb{A}/E,O(\bar{x}_{1},\bar{x}_{2})\rangle $ is dense and we have $ \langle\mathbb{A}/E,O(\bar{x}_{1},\bar{x}_{2})\rangle\cong \langle\mathbb{Q},<\rangle. $
\end{proof}

The lemma above answers the question proposed in \cite{Evans&Wong-SomeRemarksonGen}. Hence the following theorem is established.
\begin{thm}\label{thmFMP}
$ M^{*} $ does not have the finite model property.\qed
\end{thm}

The following result enables us to interpret true arithmetic in $ M^{*}. $
\begin{prop}\label{propQSumMul}
The structure $ \langle\mathbb{Q}^{\geq 0},+,.,<\rangle $ is interpretable in $ M^{*}. $
\end{prop}
\begin{proof}
Sticking to the notations used in the proof of \Cref{lmaQOrder} we define multiplication on $ \mathbb{A}/E $ by the formula $ M(\bar{x}_{1},\bar{x}_{2},\bar{x}_{3}) $ which is defined as follows
\begin{align*}
\bigwedge_{i=1}^{3}\phi_{A}(\bar{x}_{i})\hspace*{2pt}\wedge\hspace*{2pt}&\exists \bar{u}\bar{v}\bar{w}\Big[\bigwedge_{i=1}^{3}\beta_{A}(\bar{x}_{i},u_{i})\wedge\bigwedge_{i=1}^{3}\gamma_{A}(\bar{x}_{i},v_{i})\wedge\\
&\pi^{1}(u_{1},u_{2},w_{1})\wedge\pi^{1}(v_{1},v_{2},w_{2})\wedge\theta^{2}(w_{1},u_{3})\wedge\theta^{2}(w_{2},v_{3})\Big].
\end{align*}
where $ \card[\bar{u}]=\card[\bar{v}]=3 $ and $ \card[\bar{w}]=2. $

Also to define the addition function, let $ A(\bar{x}_{1},\bar{x}_{2},\bar{x}_{3}) $ be the following formula
\begin{align*}
\bigwedge_{i=1}^{3}\phi_{A}(\bar{x}_{i})\wedge\exists \bar{u}\bar{v}\bar{w}s&\Big[\bigwedge_{i=1}^{3}\beta_{A}(\bar{x}_{i},u_{i})\wedge\bigwedge_{i=1}^{3}\gamma_{A}(\bar{x}_{i},v_{i})\wedge\pi^{1}(u_{1},v_{2},w_{1})\wedge\pi^{1}(v_{1},u_{2},w_{2})\\
&\wedge\pi^{1}(v_{1},v_{2},w_{3})\wedge\upsilon^{2}(w_{1},w_{2},s)\wedge\theta^{2}(s,u_{3})\wedge\theta^{2}(w_{3},v_{3})\Big].
\end{align*}
where $ \card[\bar{u}]=\card[\bar{v}]=\card[\bar{w}]=3. $
\end{proof}

\begin{crl}\label{crlTrueArithmetics}
The structure $ \langle\mathbb{N},<,+,.,1\rangle $ is interpretable in $ M^{*}. $ In particular, true arithmetic is interpretable in $ \theory(M^{*}). $\qed
\end{crl}

\begin{paragraph}{Acknowledgement.}
\hspace*{-7pt}The first author greatly benefited from discussions on this paper with John Baldwin and his comments. We would like to thank Zaniar Ghadernezhad for some helpful discussions. Also we would like to thank Mohsen Khani for carefully reading the paper and giving useful comments and suggestions.
\end{paragraph}

\bibliographystyle{alpha}

\end{document}